\documentclass[11pt,a4paper]{article}
\date{}
\makeatletter
\@addtoreset{equation}{section}
\makeatother
\usepackage[colorlinks,
            linkcolor=blue,
            anchorcolor=blue,
            citecolor=blue]{hyperref}
\usepackage{hyperref}
\hypersetup{hypertex=true,
            colorlinks=true,
            linkcolor=blue,
            anchorcolor=blue,
            citecolor=blue}
\usepackage{hyperref,amsmath,amssymb,amscd}
\usepackage{amssymb,amsmath,enumerate}
\usepackage{indentfirst}
\usepackage{url}
\usepackage[capitalise]{cleveref}
\usepackage{amsthm}
\usepackage{cite}
\usepackage{appendix}
\newtheorem{thm}{Theorem}[section]
\newtheorem{cor}[thm]{Corollary}
\newtheorem{lem}[thm]{Lemma}
\newtheorem{prop}[thm]{Proposition}
\theoremstyle{definition}
\newtheorem{defn}[thm]{Definition}

\newtheorem{rem}[thm]{\bf Remark}
\newtheorem{\appendixname}{\Appendix~Alph}
\allowdisplaybreaks[4]

\begin{document}

\title{\bf
Existence and asymptotic behavior of least energy sign-changing
solutions for Schr\"{o}dinger-Poisson systems with doubly
critical exponents\footnote{Supported by National Natural Science
Foundation of China (No. 11971393).}}
\author{{Xiao-Ping Chen,\ Chun-Lei Tang\footnote{Corresponding
author.
E-mail address: tangcl@swu.edu.cn (C.-L. Tang).}}\\
{\small \emph{ School of Mathematics and Statistics, Southwest
University,  Chongqing {\rm400715},}}\\
{\small \emph{People's Republic of China}}}
\maketitle
\baselineskip 17pt

\noindent {\bf Abstract}:\ In this paper, we are concerned with the
following Schr\"{o}dinger-Poisson system with critical
nonlinearity and critical nonlocal term due to the
Hardy-Littlewood-Sobolev inequality
\begin{equation*}
\begin{cases}
 -\Delta u+u+\lambda\phi |u|^3u =|u|^4u+ |u|^{q-2}u,\ \ &\ x \in
 \mathbb{R}^{3},\\[2mm]
  -\Delta \phi=|u|^5, \ \ &\ x \in \mathbb{R}^{3},
\end{cases}
\end{equation*}
where $\lambda\in \mathbb{R}$ is a parameter and $q\in(2,6)$. If
$\lambda\ge (\frac{q+2}{8})^2$ and $q\in(2,6)$, the above system
has no nontrivial solution. If $\lambda\in (\lambda^*,0)$ for some $\lambda^*<0$, we obtain a least energy radial sign-changing
solution $u_\lambda$ to the above system. Furthermore, we consider $\lambda$ as a parameter and analyze the asymptotic behavior of $u_\lambda$ as
$\lambda\to 0^-$.


\vspace{0.125em}
\vspace*{0.125em}


\par

\section{Introduction and main results}

\noindent In this article, we are interested in the existence,
nonexistence and asymptotic behavior of least energy
sign-changing solutions for the following Schr\"{o}dinger-Poisson
system
\begin{equation}\tag{$\mathcal{{SP}}$}\label{1.6}
\begin{cases}
 -\Delta u+u+\lambda\phi |u|^3u =|u|^4u+ |u|^{q-2}u,\ \ &\ x \in
 \mathbb{R}^{3},\\[2mm]
  -\Delta \phi=|u|^5, \ \ &\ x \in \mathbb{R}^{3},
\end{cases}
\end{equation}
where $\lambda \in \mathbb{R}$ is a parameter and $q\in(2,6)$.

When it comes to the following more general cases of
Schr\"{o}dinger-Poisson system
\begin{equation}\label{11}
\begin{cases}
 -\Delta u+V(x)u+\lambda \phi(x) |u|^{s-2}u = f(u),\ \ &\ x \in
 \mathbb{R}^{3},\\[2mm]
 -\Delta \phi=|u|^{s}, \ \ &\ x \in \mathbb{R}^{3},\\
\end{cases}
\end{equation}
which has been studied extensively, where $\lambda\in \mathbb{R}$ is a parameter and $s\in[\frac{5}{3},5]$, the potential function
$V(x)$ is continuous and the nonlinearity $f$ satisfies some suitable assumptions. Notice that the numbers $\frac{5}{3}$ and $5$ are called the lower and the upper critical exponents due to the Hardy-Littlewood-Sobolev
inequality (see Proposition \ref{prohlsi} below), respectively.
In the past decades, much attention has been attracted to system
\eqref{11}  with $s\in[\frac{5}{3},5)$ (i.e., the subcritical
nonlocal case), see for example, \cite{LFY13,AA13}.

Especially, for the case of $s\equiv 2$, system \eqref{11} turns
to the following system
\begin{equation}\label{22}
\begin{cases}
 -\Delta u+V(x)u+\lambda \phi(x) u = f(u),\ \ &\ x \in
 \mathbb{R}^{3},\\[2mm]
 -\Delta \phi=|u|^{2}, \ \ &\ x \in \mathbb{R}^{3}.\\
\end{cases}
\end{equation}
System \eqref{22} was firstly proposed by Benci and Fortunato
\cite{BV98} in 1998 to describe the interaction between a charged
particle and the electrostatic field in quantum mechanics. For
more detailed physical background of system \eqref{22}, we refer
the readers to \cite{AA08,BV02} and the reference therein. Due to
its physical significance, which has been investigated by many
researchers for the existence and nonexistence of positive
solutions, multiple solutions, least energy solutions, radial
solutions, sign-changing solutions and so on, see for example,
\cite{CG10,CXP21,CXP22,RD06,SW15,WDB19,WZP15,LFY20,ZJ15,ZZH21,ZXJ18}.
In \cite{CG10,RD06}, the existence of positive solutions was
investigated for the subcritical nonlinearity $f$. The existence
of least energy solutions for system \eqref{22} has been
considered in \cite{SJT12,LFY20}. Moreover, the existence of
least energy sign-changing solutions for system \eqref{22} with
subcritical nonlinearity was investigated in \cite{SW15,WZP15},
and the existence of least energy sign-changing solutions for
system \eqref{22} with critical nonlinearity has been considered
in \cite{CXP21,CXP22,WDB19,ZJ15,ZZH21,ZXJ18}.

However, all the papers mentioned above investigated the
subcritical nonlocal term. To the best of our knowledge,
Schr\"{o}dinger-Poisson system with critical nonlocal term (i.e.,
$s\equiv 5$) has only been studied in
\cite{HXM21,LFY14,LFY17,LH16,AA12}. Azzollini and d'Avenia
\cite{AA12} firstly studied the Schr\"{o}dinger-Poisson system
with critical nonlocal term as follows:
\begin{equation*}
\begin{cases}
 -\Delta u=\lambda u+q \phi |u|^{3}u,\ \ &\ x \in B_R,\\[1mm]
 -\Delta \phi=q|u|^{5}, \ \ &\ x \in B_R,\\[1mm]
 u=\phi=0, \ \ &\ \mathrm{on}\ \partial B_R,\\
\end{cases}
\end{equation*}
where $B_R(0)\subset \mathbb{R}^3$ is a ball centered in $0$ with
radius $R$. The authors considered the existence and nonexistence of
nontrivial solutions involving the range of $\lambda$. After
this, many researchers are devoted to investigating the
Schr\"{o}dinger-Poisson system with critical nonlocal term. As we know, compared with the
Schr\"{o}dinger-Poisson system with subcritical nonlocal term, there are few results to investigate
the Schr\"{o}dinger-Poisson system with critical nonlocal term.

In \cite{LH16}, by concentration-compactness principle, Liu
proved the existence of positive solutions for the following
system
\begin{equation}\label{33}
\begin{cases}
 -\Delta u+V(x)u-K(x) \phi(x) |u|^{3}u = f(u),\ \ &\ x \in
 \mathbb{R}^{3},\\[2mm]
 -\Delta \phi=K(x)|u|^{5}, \ \ &\ x \in \mathbb{R}^{3},\\
\end{cases}
\end{equation}
where $V,\,K$ and $f$ are asymptotically periodic functions with
respect to $x$. Li, Li and Shi \cite{LFY14} investigated the
existence of positive radial symmetric solutions for system
\eqref{33} with $V(x)\equiv b$ (a positive constant) and
$K(x)\equiv 1$, by using variational methods without usual
compactness conditions. Later, in \cite{LFY17}, they also paid
attention to the existence, nonexistence and multiplicity of
positive radial solutions for the following system
\begin{equation*}
\begin{cases}
 -\Delta u+u+\lambda \phi|u|^{3}u = \mu |u|^{p-1}u,\ \ &\ x \in
 \mathbb{R}^{3},\\[2mm]
 -\Delta \phi=|u|^{5}, \ \ &\ x \in \mathbb{R}^{3},\\
\end{cases}
\end{equation*}
where $\mu\ge 0$ and $\lambda\in \mathbb{R}$ are parameters.
Recently, in \cite{HXM21}, He considered a fractional
Schr\"{o}dinger-Poisson system with critical nonlocal term, and
the system studied there was as follows:
\begin{equation}\label{44}
\begin{cases}
 (-\Delta)^s u+V(x)u- \phi(x) |u|^{2_s^*-3}u =
 |u|^{2_s^*-2}u+f(u),\ \ &\ x \in \mathbb{R}^{3},\\[2mm]
 (-\Delta)^s \phi=|u|^{2_s^*-1}, \ \ &\ x \in \mathbb{R}^{3},\\
\end{cases}
\end{equation}
where $s\in(0,1)$ and $V(x)$ is a coercive potential, the author
proved that system \eqref{44} possesses at least one positive
solutions by employing mountain pass theorem.

Motivated by the above works, especially the results obtained in
\cite{LFY14,LFY17,LH16,HXM21}, in this paper, we investigate the
existence of least energy sign-changing solutions for system
\eqref{1.6} with $\lambda\in \mathbb{R}$, which involves critical
nonlinearity and critical nonlocal term. As far as we know, in
the case of $\lambda>0$, it seems that there is no result about
the existence of nontrivial solutions for system \eqref{1.6}.
Here we consider it. In what follows, we state the nonexistence
result of nontrivial solutions for system \eqref{1.6}.


\begin{thm}\label{thm-3}
For $\lambda\ge (\frac{q+2}{8})^2$ and $q\in(2,6)$, then system
\eqref{1.6} has no nontrivial solution.
\end{thm}

\begin{rem}
As has been pointed out above, Theorem \ref{thm-3} is the first
attempt to investigate the nonexistence of nontrivial solutions
for system \eqref{1.6} with critical nonlinearity and critical
nonlocal term. We will give the proof of Theorem \ref{thm-3} with
the help of Young's inequality.
\end{rem}

\begin{rem}
The existence or nonexistence results for nontrivial solutions of
system \eqref{1.6} are completely different from the system with
subcritical nonlocal term. In \cite{ZZH21}, authors obtained a
least energy  sign-changing solution for the following system
\begin{equation*}
\begin{cases}
 -\Delta u+u+\phi u =|u|^4u+ |u|^{q-2}u,\ \ &\ x \in
 \mathbb{R}^{3},\\[2mm]
  -\Delta \phi=|u|^2, \ \ &\ x \in \mathbb{R}^{3},
\end{cases}
\end{equation*}
where $q\in(5,6)$. However, Theorem \ref{thm-3} shows that system
\eqref{1.6} has no nontrivial solution for $\lambda\ge
(\frac{q+2}{8})^2$ and $q\in(2,6)$. Thereby, this indicates the
difference between the Schr\"{o}dinger-Poisson system with a
critical nonlocal term and that with a subcritical nonlocal term.
\end{rem}

Now, we turn to investigate the existence of sign-changing
solutions for system \eqref{1.6}. Inspired by the papers
mentioned above, there is no result to study the existence of
sign-changing solutions for system \eqref{1.6} which involves
critical nonlinearity and a critical nonlocal term, a natural
question is whether system \eqref{1.6} possesses a least energy
sign-changing solution or not? In the following theorem, we will
give an affirmative answer to this question. The
Schr\"{o}dinger-Poisson system with critical nonlinearity and
critical nonlocal term is much more difficult to obtain the
existence of sign-changing solutions. The first difficulty is the
lack of compactness. The second difficulty is the competition
between the critical nonlocal term and the critical nonlinearity.
Now, we are in a position to state the existence result of least
energy sign-changing solutions for system \eqref{1.6}. In the
following results, we will work in the space
$H_r^1(\mathbb{R}^3)$ which contains all radial symmetric
functions of standard Hilbert space $H^1(\mathbb{R}^3)$.

In the following, we consider the case of $\lambda<0$, which is a
quite crucial result of this paper.

\begin{thm}\label{thm-2}
For $q\in(5,6)$, there exists $\lambda^*<0$ such that for all
$\lambda\in(\lambda^*,0)$, then system \eqref{1.6} possesses at least one least energy radial sign-changing solution $u_\lambda$.
\end{thm}

\begin{rem}
Theorem \ref{thm-2} is the first existence result about least
energy sign-changing solutions for system \eqref{1.6} which
involves critical nonlinearity and critical nonlocal term. Here we would like to emphasize that the existence of
sign-changing solutions for Schr\"{o}dinger-Poisson system with
subcritical nonlocal term has been considered in many papers for
the case $\lambda>0$ not the case $\lambda<0$. It is worthwhile
noticing that the existence of sign-changing solutions for system
\eqref{1.6} with critical nonlocal term has not been investigated
in either $\lambda>0$ or $\lambda<0$ case. Here we consider the
both cases.
\end{rem}

\begin{rem}
Compared with \cite{CXP21,CXP22,WDB19,ZZH21,ZJ15,ZXJ18}, where
 authors investigated a class of Schr\"{o}dinger-Poisson
system with critical nonlinearity and subcritical nonlocal term,
we investigate the existence of least energy sign-changing
solutions for Schr\"{o}dinger-Poisson system with doubly critical
exponents. Instead of the single critical exponent, the presence
of doubly critical exponents make it more difficult to recover
the compactness condition and estimate energy.
\end{rem}

Now, we give the main ideas for the proof of Theorem \ref{thm-2}.
Since our main purpose is to investigate the existence of
sign-changing solutions for system \eqref{1.6} involving critical
nonlinearity, the first thing we need to do is to construct a
minimizing Palais-Smale sequence ($(\mathrm{PS})$ sequence for
short) on the sign-changing Nehari manifold, we borrow some ideas
from \cite{CG36} to obtain it, see Lemma \ref{lam-76} below. The
second thing is to prove the $(\mathrm{PS})_{m_\lambda}$
condition with the help of the upper bound of the least
energy $m_\lambda$ on the sign-changing Nehari manifold
$\mathcal{M}_\lambda$. It is worth mentioning that to estimate
the least energy $m_\lambda$ on the sign-changing Nehari manifold
$\mathcal{M}_\lambda$ is a key point in the proof of Theorem
\ref{thm-2}, we established some more subtle estimates to obtain
it.


Note that when $\lambda\equiv0$, system
\eqref{1.6} reduces to the following equation
\begin{equation}\label{1.7}
  -\Delta u+u =|u|^4u+ |u|^{q-2}u,\ \ \ x \in \mathbb{R}^{3}.
\end{equation}
Then according to Theorem \ref{thm-2}, we obtain the following result.

\begin{cor}\label{thm-5}
For $q\in(5,6)$, then problem \eqref{1.7} possesses at least one
least energy radial sign-changing solution with exactly two nodal
domains, and its energy is larger than twice that of least energy
radial solutions.
\end{cor}

\begin{rem}
In view of Theorem \ref{thm-3}, Theorem \ref{thm-2} and Corollary \ref{thm-5}, it is still an
open question to study the existence of nontrivial solutions for
the case $\lambda\in(0,(\frac{q+2}{8})^2)$, if nontrivial
solutions exist, whether least energy sign-changing solutions
exist. Hence, it is worth exploring a new technique to study the
existence of nontrivial solutions of system \eqref{1.6} for the
case $\lambda\in(0,(\frac{q+2}{8})^2)$.

\end{rem}

We further study the asymptotic behavior of the least energy
radial sign-changing solution $u_\lambda$ obtained in Theorem
\ref{thm-2} as $\lambda\rightarrow 0^-$, which indicates the
relationship between $\lambda<0$ and $\lambda=0$ in system
\eqref{1.6}.

\begin{thm}\label{thm-4}
Let $u_{\lambda_n}$ be a least energy radial sign-changing
solution of system \eqref{1.6} with $\lambda=\lambda_n$ obtained
in Theorem \ref{thm-2}, then for any sequence $\{\lambda_n\}$
with ${\lambda_n\rightarrow {0^-}}$ as $n\rightarrow\infty$,
there exists a subsequence, still denoted by $\{\lambda_n\}$,
such that ${u_{\lambda_n}}$ converges to $u_0$ weakly in
$H_r^1(\mathbb{R}^3)$ as $n\to \infty$, where $u_0$ is a least
energy radial sign-changing solution of problem \eqref{1.7} which
has precisely two nodal domains.
\end{thm}

\begin{rem}
To out knowledge, this paper is the first attempt to
prove the existence and asymptotic behavior of least energy
radial sign-changing solutions for system \eqref{1.6}.

Moreover, for proving the asymptotic behavior of least energy
sign-changing solutions of system \eqref{1.6} like Theorem
\ref{thm-4}, we refer the interested readers to
\cite{WDB19,CXP21}. However, the methods of analyzing the
asymptotic behavior used in \cite{WDB19,CXP21} heavily depend on
the following inequality:
\begin{align*}
0<\rho< \|u^\pm_{\lambda_n}\|^2
\le \int_{\mathbb{R}^{3}}|u^\pm_{\lambda_n}|^6\mathrm{d}x
+\mu\int_{\mathbb{R}^{3}}
f(u^\pm_{\lambda_n})u^\pm_{\lambda_n}\mathrm{d}x
\le 2\mu\int_{\mathbb{R}^{3}}
f(u^\pm_{\lambda_n})u^\pm_{\lambda_n}\mathrm{d}x,
\end{align*}
for $\mu$ large enough, which is used to prove $u_0^\pm\neq 0$.
Therefore, they only analyzed the asymptotic behavior of least
energy sign-changing solutions for Schr\"{o}dinger-Poisson system
which does not involve a critical nonlinearity, which indicates
that the method of analyzing the asymptotic behavior in
\cite{WDB19,CXP21} is not applicable to our paper. Since problem
\eqref{1.7} involves a critical nonlinearity, the main difficulty
we encounter in the proof of Theorem \ref{thm-4} is to obtain
$u_0^\pm\neq 0$. Here we will use a method of proof by
contradiction and some technical analysis to overcome it.

\end{rem}

This paper is constructed as follows. In Section \ref{2}, we
present the variational framework. In Section \ref{3}, we prove
Theorem \ref{thm-3} for the case $\lambda\ge
(\frac{q+2}{8})^2$. In Section
\ref{5}, we complete the proof of Theorem \ref{thm-2} for the case $\lambda<0$. Section \ref{4} is interested in proving
Theorem \ref{thm-5} for the case $\lambda=0$. In Section \ref{6}, we analyze the asymptotic behavior of the least energy radial sign-changing solution $u_\lambda$ obtained in Theorem \ref{thm-2} as
$\lambda\rightarrow 0^-$, and prove Theorem \ref{thm-4}.

\section{Preliminaries}\label{2}

\noindent In this section, we provide some notations, work space
stuff and present some useful propositions which are crucial for
proving our main results. We firstly present some necessary
notations which will be used throughout this paper.

\begin{itemize}
\setlength{\itemsep}{0pt}
\setlength{\parsep}{0pt}
\setlength{\parskip}{0pt}
\setlength{\itemindent}{1em}\item $``\rightharpoonup"$
($``\rightarrow"$) denotes the weak (strong) convergence.
\setlength{\itemindent}{1em}\item
is the norm in the usual Lebesgue space $L^p({\mathbb{R}^{3}})$
for $p\in[1,+\infty)$.
$L^p({\mathbb{R}^{3}})$ is the usual Lebesgue space with the
norm
$$|u|_p
=\left(\int_{\mathbb{R}^{3}}|u|^{p}\mathrm{d}x\right)^{\frac{1}{p}}
\ \ \mathrm{for\ all}\ p\in[1,\infty),\ \ \mathrm{and}\
|u|_\infty=\mathrm{ess}\,\sup\limits_{x\in
\mathbb{R}^3}|u(x)|.$$
\setlength{\itemindent}{1em}\item Denote
$\mathcal{D}^{1,2}(\mathbb{R}^{3}):=\left\{u\in
L^6(\mathbb{R}^{3}):|\nabla u|\in L^2(\mathbb{R}^3)\right\}$
equipped with the norm
$$\|u\|_{\mathcal{D}^{1,2}(\mathbb{R}^{3})}
:=\left(\int_{\mathbb{R}^{3}}|\nabla
u|^2\mathrm{d}x\right)^{\frac{1}{2}}.$$
\setlength{\itemindent}{1em}\item Let
$H^1(\mathbb{R}^{3}):=\left\{u\in L^2(\mathbb{R}^{3}):|\nabla
u|\in L^2(\mathbb{R}^3)\right\}$ be the Hilbert space endowed
with the inner product and  norm
$$\langle u,v\rangle=\int_{\mathbb{R}^{3}}\left(\nabla u \cdot
\nabla v+uv\right)\mathrm{d}x,\ \ \
\|u\|=\left(\int_{\mathbb{R}^{3}}(|\nabla
u|^2+u^2)\mathrm{d}x\right)^{\frac{1}{2}}.$$
\setlength{\itemindent}{1em}\item $\mathcal{C}_0^\infty(\mathbb{R}^3)$ contains infinitely times differentiable functions with compact support in $\mathbb{R}^3$.
\setlength{\itemindent}{1em}\item $o(1)$ denotes a quantity
which goes to $0$ as $n\rightarrow\infty$.
\setlength{\itemindent}{1em}\item $O(\varepsilon)$ denotes a
bounded quantity as $\varepsilon\rightarrow 0$.
\setlength{\itemindent}{1em}\item $C,\ C_i\ (i\in
\mathbb{N}^+)$ denote various positive constants.
\end{itemize}

Let $S$ be the best Sobolev constant for the embedding
$\mathcal{D}^{1,2}(\mathbb{R}^{3})\hookrightarrow
L^6(\mathbb{R}^{3})$, that is,
\begin{equation}\label{2.3}
S:
=\inf\limits_{{u\in\mathcal{D}^{1,2}}(\mathbb{R}^{3})\backslash\{0\}}
\frac{\int_{\mathbb{R}^{3}}|\nabla u|^2\mathrm{d}x}
{\left(\int_{\mathbb{R}^{3}}|u|^{6}\mathrm{d}x\right)^{\frac{1}{3}}}.
\end{equation}
By \cite[Theorem 1.42]{WM14}, $S$ is attained by the following
function
\begin{equation}\label{2.4}
\xi_{\varepsilon}(x):=\frac{
\varepsilon^{\frac{1}{4}}}{(\varepsilon+|x|^2)^{\frac{1}{2}}},
\end{equation}
for any $\varepsilon>0$ and $x\in \mathbb{R}^3$, which satisfies
that
$\|\xi_\varepsilon\|_{\mathcal{D}^{1,2}(\mathbb{R}^{3})}^2
=|\xi_\varepsilon|_6^6=S^{\frac{3}{2}}$. Moreover, Let us define
a subspace of $H^1(\mathbb{R}^3)$ which contains all radial
symmetric functions as 
\begin{equation*}
H_r^{1}(\mathbb{R}^3):=\left\{u\in H^{1}(\mathbb{R}^3):
u(x)=u(|x|)\right\}.
\end{equation*}

Next, we would like to show the following embedding proposition
which will be used frequently.

\begin{prop}[See {{\cite[Theorem B]{WC20}}}]\label{lam-32}
The embedding $H_r^{1}(\mathbb{R}^3)\hookrightarrow
L^r({\mathbb{R}^{3}})$ is continuous for $r\in[2,6]$. Moreover,
the embedding $H_r^{1}(\mathbb{R}^3)\hookrightarrow
L^r({\mathbb{R}^{3}})$ is compact for $r\in[2,6)$.
\end{prop}

The well known Hardy-Littlewood-Sobolev inequality presented as
follows which plays a crucial role in constructing the
variational framework for system \eqref{1.6}.

\begin{prop}[{\bfseries Hardy-Littlewood-Sobolev inequality, see}
{{\cite{LE83,LE01}}}]\label{prohlsi}
    Let $r,t>1$ and $0<\alpha<3$ with
    $\frac{1}{r}+\frac{1}{t}+\frac{\alpha}{3}=2$, $u\in
    L^r(\mathbb{R}^3)$ and $v\in L^t(\mathbb{R}^3)$. Then, there
    exists a sharp constant $C(r,t,\alpha)>0$ (independent of $u$
    and $v$) such that
    \begin{equation}\label{hlsi}
    \int_{\mathbb{R}^3}\int_{\mathbb{R}^3}
    \frac{u(x)v(y)}{|x-y|^\alpha}
    \mathrm{d}x \mathrm{d}y
    \leq C(r,t,\alpha)|u|_{r}|v|_{t}.
    \end{equation}
    In particular, if $r=t=\frac{6}{6-\alpha}$, then
    \begin{equation*}
    C(r,t,\alpha)
    =C(\alpha)
    =\frac{\Gamma((3-\alpha)/2)\pi^{\alpha/2}}{\Gamma(3-\alpha/2)}
    \left(\frac{\Gamma(3)}{\Gamma(3/2)}\right)^{(3-\alpha)/3},
    \end{equation*}
    and there is equality in (\ref{hlsi}) if and only if $u\equiv
    (const.) v$ and
    \begin{equation*}
    v(x)=A\left(1+\lambda^2|x-z|^2\right)^{-
    \frac{6-\alpha}{2}}
    \end{equation*}
    for some given constants $A\in \mathbb{C}$, $\lambda\in
    \mathbb{R}\setminus\{0\}$ and for some point $z\in
    \mathbb{R}^3$.
\end{prop}

    Motivated by Proposition \ref{prohlsi}, for any $u\in
    H_r^1(\mathbb{R}^3)$, we know that
    \begin{equation}\label{6.96}
    \int_{\mathbb{R}^3}\int_{\mathbb{R}^3}
    \frac{|u(x)|^s|u(y)|^s}{|x-y|}
    \mathrm{d}x \mathrm{d}y
    \end{equation}
    is well-defined if for $u\in L^{st}(\mathbb{R}^3)$ satisfying
    \[\frac{2}{t}+\frac{1}{3}=2,\]
    $i.e.$, $t=\frac{6}{5}$.
    Then, the Sobolev continuous embedding
    $H_r^{1}(\mathbb{R}^3)\hookrightarrow L^p(\mathbb{R}^3)$ for
    $p\in \left[2,6\right]$ shows that
    \begin{equation*}
    2\le st\le 6\Rightarrow\frac{5}{3}\leq s\leq 5.
    \end{equation*}
    Thus, \eqref{6.96} is well-defined if $s\in[\frac{5}{3},5]$.
    The numbers $\frac{5}{3}$ and $5$ are called the lower and
    the upper critical exponents with respect to the
    Hardy-Littlewood-Sobolev inequality, respectively.

For a given $u\in H^{1}(\mathbb{R}^3)$, Lax-Milgram theorem
implies that there exists a unique $\phi_u\in
\mathcal{D}^{1,2}(\mathbb{R}^{3})$ such that $-\Delta
\phi=|u|^{5}$ in a weak sense. Moreover, according to
\cite[Theorem 6.21]{LE01}, we see that
\begin{equation}\label{2.55}
\phi_u(x)
=\frac{1}{4\pi}
\int_{\mathbb{R}^{3}}\frac{|u(y)|^{5}}{|x-y|}\mathrm{d}y> 0,
\end{equation}
and
\begin{equation*}
\int_{\mathbb{R}^{3}}\phi_{u}|u|^5\mathrm{d}x
=\frac{1}{4\pi}\int_{\mathbb{R}^{3}}\int_{\mathbb{R}^{3}}
\frac{|u(x)|^{5}|u(y)|^{5}}{|x-y|}\mathrm{d}x\mathrm{d}y.
\end{equation*}
It derives from the above equality and Fubini theorem that
\begin{equation*}
\int_{\mathbb{R}^{3}}\phi_{u^+}|u^-|^5\mathrm{d}x
=\int_{\mathbb{R}^{3}}\phi_{u^-}|u^+|^5\mathrm{d}x>0.
\end{equation*}

Next, we list some properties of $\phi_u$, which can directly
deduce from \cite[Lemma 2.1]{LFY14}.

\begin{prop}\label{prop-1}
For any $u\in H_r^{1}(\mathbb{R}^3)$, $-\Delta \phi=|u|^{5}$
possesses a unique weak solution $\phi_u\ge 0$ in
$\mathcal{D}^{1,2}(\mathbb{R}^3)$. Moreover, there hold

  (1) $\|\phi_u\|_{\mathcal{D}^{1,2}(\mathbb{R}^{3})}^2
  =\int_{\mathbb{R}^{3}}\phi_{u}|u|^5\mathrm{d}x$;

  (2) $\phi_{su}=s^{5}{\phi_u}$, for any $s>0$;

  (3) $\|\phi_{u}\|_{\mathcal{D}^{1,2}(\mathbb{R}^{3})}\le C_1
  \|u\|^{5}$,
  $\int_{\mathbb{R}^{3}}\phi_{u}|u|^5\mathrm{d}x\le
  C_2\|u\|^{10}$ for some  $C_1,C_2>0$ (independent of $u$);

  (4) if $u$ is a radial function, and $\phi_u$ is also radial;

  (5) if $u_n\rightharpoonup u$ in $H_r^{1}(\mathbb{R}^3)$ and
  $u_n\to u$ a.e. in $\mathbb{R}^3$, then
  $\phi_{u_n}\rightharpoonup \phi_u$ in
  $\mathcal{D}^{1,2}(\mathbb{R}^{3})$ and
  $$\int_{\mathbb{R}^{3}}\phi_{u_n} |u_n|^{5}\mathrm{d}x
  -\int_{\mathbb{R}^{3}}\phi_{u_n-u} |u_n-u|^{5}\mathrm{d}x
  =\int_{\mathbb{R}^{3}}\phi_{u} |u|^{5}\mathrm{d}x+o(1).$$
\end{prop}

\noindent Substituting \eqref{2.55} into system \eqref{1.6}, we
rewrite system \eqref{1.6} into the following equation
\begin{equation*}
  -\Delta u+u+\lambda\phi_u |u|^3u =|u|^4u+ |u|^{q-2}u,\ \ x \in
  \mathbb{R}^{3}.
\end{equation*}
Next, we define the energy functional
$\mathcal{I}_\lambda:H^{1}(\mathbb{R}^3)\rightarrow \mathbb{R}$
corresponding to system \eqref{1.6} by
\begin{equation*}
\mathcal{I}_\lambda(u)=\frac{1}{2}\|u\|^2
+\frac{\lambda}{10}\int_{\mathbb{R}^{3}}\phi_u |u|^5\mathrm{d}x
-\frac{1}{6}\int_{\mathbb{R}^{3}}|u|^{6}\mathrm{d}x
-\frac{1}{q}\int_{\mathbb{R}^{3}}|u|^{q}\mathrm{d}x.
\end{equation*}
Clearly, $\mathcal{I}_\lambda$ is well-defined in
$H^{1}(\mathbb{R}^3)$ and of class $\mathcal{C}^{1}$ for any $\lambda\in \mathbb{R}$, and for any $u,v\in H^{1}(\mathbb{R}^3)$,
\begin{equation}\label{2.7}
\langle \mathcal{I}_\lambda'(u),v\rangle
=\int_{\mathbb{R}^{3}}(\nabla u\cdot\nabla v+uv)\mathrm{d}x
+\lambda\int_{\mathbb{R}^{3}}\phi_u |u|^3u v\mathrm{d}x
-\int_{\mathbb{R}^{3}}|u|^{4}uv\mathrm{d}x
-\int_{\mathbb{R}^{3}}|u|^{q-2}uv\mathrm{d}x.
\end{equation}
Thereby, we can directly use the variational methods to
investigate the nontrivial weak solutions of system \eqref{1.6}
by studying the critical points of the functional
$\mathcal{I}_\lambda$ in $H^1(\mathbb{R}^3)$.

In what follows, we recall some definitions of critical points
for the functional $\mathcal{I}_\lambda$.

\begin{defn}\label{def-6}

(1) If $u\in H^{1}(\mathbb{R}^3)$ such that $\langle
\mathcal{I}_\lambda'(u),v\rangle=0$ for any $v\in
H^{1}(\mathbb{R}^3)$, then we say that $u$ is a weak solution of
system \eqref{1.6}.

(2) A nontrivial solution $u$ of \eqref{1.6} is called a ground
state (or least energy) solution if
\begin{equation}\label{9.25}
\mathcal{I}_\lambda(u)
=c_\lambda:
=\inf\limits_{v\in\mathcal{N}_\lambda}\mathcal{I}_\lambda(v),
\end{equation}
where
\begin{equation*}
\mathcal{N}_\lambda:=\left\{u\in
H_r^1(\mathbb{R}^3)\setminus\{0\}:\langle
\mathcal{I}_\lambda'(u),u\rangle=0\right\}.
\end{equation*}

(3) If $u\in H^{1}(\mathbb{R}^3)$ is a weak solution of system
\eqref{1.6} with $u^\pm\neq0$, then $u$ is a sign-changing
solution of system \eqref{1.6}, where $u^+:=\mathrm{max}\{u,0\}\
\mathrm{and}\ u^-:=\mathrm{min}\{u,0\}.$

(4) A sign-changing solution $u$ of \eqref{1.6} is called a least
energy sign-changing solution if
\begin{equation}\label{9.26}
\mathcal{I}_\lambda(u)
=m_\lambda:
=\inf\limits_{v\in \mathcal{M}_\lambda} \mathcal{I}_\lambda(v),
\end{equation}
where
\begin{equation*}
\mathcal{M}_\lambda:=\left\{u\in
H_r^1(\mathbb{R}^3):u^{\pm}\neq0,\ \langle
\mathcal{I}_\lambda'(u),u^+\rangle=\langle
\mathcal{I}_\lambda'(u),u^-\rangle=0\right\}.
\end{equation*}
\end{defn}


\section{Nonexistence of nontrivial solutions for the case
$\lambda\ge (\frac{q+2}{8})^2$}
\label{3}
\noindent In this section, we are ready to prove the nonexistence
of nontrivial solutions for system \eqref{1.6} for the case of
$\lambda\ge (\frac{q+2}{8})^2$, and prove Theorem \ref{thm-3}.

\begin{proof}[\textbf{Proof of Theorem \ref{thm-3}}]
By using $q\in(2,6)$ and Young's inequality, we see that
\begin{equation}\label{66}
\int_{\mathbb{R}^{3}}|u|^q\mathrm{d}x
\le\frac{6-q}{4}\int_{\mathbb{R}^{3}}|u|^{2}\mathrm{d}x
+\frac{q-2}{4}\int_{\mathbb{R}^{3}}|u|^{6}\mathrm{d}x.
\end{equation}
By $-\Delta \phi_u=|u|^5$, it holds that
\begin{align*}
\int_{\mathbb{R}^3}|u|^6\mathrm{d}x
=\int_{\mathbb{R}^3}-\Delta \phi_{u} |u|\mathrm{d}x
&\le \frac{1}{2\tau^2}\int_{\mathbb{R}^3}|\nabla
\phi_{u}|^2\mathrm{d}x
+\frac{\tau^2}{2}\int_{\mathbb{R}^3}|\nabla |{u}||^2\mathrm{d}x
\nonumber\\&\leq
\frac{1}{2\tau^2}\int_{\mathbb{R}^3}\phi_{u}|u|^5\mathrm{d}x
+\frac{\tau^2}{2}\int_{\mathbb{R}^3}|\nabla {u}|^2\mathrm{d}x,
\end{align*}
for any $\tau\in\mathbb{R}\setminus\{0\}$. Since $\lambda\ge
(\frac{q+2}{8})^2>0$, taking $\tau^2=\lambda^{-\frac{1}{2}}$,
then
\begin{align}\label{6.2}
\int_{\mathbb{R}^3}\phi_{u}|u|^5\mathrm{d}x
\ge 2\lambda^{-\frac{1}{2}} \int_{\mathbb{R}^3}|{u}|^6\mathrm{d}x
-\lambda^{-1}\int_{\mathbb{R}^3}|\nabla {u}|^2\mathrm{d}x.
\end{align}
If $(u,\phi_u)\in
H^1(\mathbb{R}^3)\times\mathcal{D}^{1,2}(\mathbb{R}^3)$ is a
solution of system \eqref{1.6}, it follows that
\begin{equation*}
\int_{\mathbb{R}^{3}}(|\nabla u|^2+u^2)\mathrm{d}x
+{\lambda}\int_{\mathbb{R}^{3}}\phi_u |u|^5\mathrm{d}x
-\int_{\mathbb{R}^{3}}|u|^{6}\mathrm{d}x
-\int_{\mathbb{R}^{3}}|u|^{q}\mathrm{d}x=0.
\end{equation*}
Then, it obtains from $\lambda\ge (\frac{q+2}{8})^2$, \eqref{66}
and \eqref{6.2} that
\begin{align*}
0&=\int_{\mathbb{R}^{3}}(|\nabla u|^2+u^2)\mathrm{d}x
+{\lambda}\int_{\mathbb{R}^{3}}\phi_u |u|^5\mathrm{d}x
-\int_{\mathbb{R}^{3}}|u|^{6}\mathrm{d}x
-\int_{\mathbb{R}^{3}}|u|^{q}\mathrm{d}x
\\&\ge\int_{\mathbb{R}^{3}}(|\nabla u|^2+u^2)\mathrm{d}x
+\lambda\left(2\lambda^{-\frac{1}{2}}
\int_{\mathbb{R}^3}|{u}|^6\mathrm{d}x
-\lambda^{-1}\int_{\mathbb{R}^3}|\nabla {u}|^2\mathrm{d}x\right)
\\&\quad-\int_{\mathbb{R}^{3}}|u|^{6}\mathrm{d}x
-\frac{6-q}{4}\int_{\mathbb{R}^{3}}|u|^{2}\mathrm{d}x
-\frac{q-2}{4}\int_{\mathbb{R}^{3}}|u|^{6}\mathrm{d}x
\\&=\frac{q-2}{4}\int_{\mathbb{R}^{3}}u^2\mathrm{d}x
+\left(2\lambda^{\frac{1}{2}}-\frac{q+2}{4}\right)
\int_{\mathbb{R}^{3}}|u|^{6}\mathrm{d}x
\\&\ge \frac{q-2}{4}\int_{\mathbb{R}^{3}}u^2\mathrm{d}x.
\end{align*}
Since $q\in (2,6)$, then $u\equiv 0$. Thus, system \eqref{1.6}
has no nontrivial solution if $\lambda\ge (\frac{q+2}{8})^2$.
\end{proof}

\begin{rem}
From the above nonexistence result, we know that system
\eqref{1.6} can only possess nontrivial solutions when $\lambda
<(\frac{q+2}{8})^2$ and $q\in(2,6).$
\end{rem}

\section{Least energy radial sign-changing solutions for the case
$\lambda<0$}\label{5}
\noindent In this section, we are devoted to investigating the
existence of least energy radial sign-changing solutions for system
\eqref{1.6} with $\lambda<0$, and prove Theorem \ref{thm-2}. This
section will be divided into four subsections: proving some preliminary lemmas which are crucial for proving Theorem \ref{thm-2},
constructing a sign-changing $(\mathrm{PS})_{m_\lambda}$ sequence
for the functional $\mathcal{I}_\lambda$, estimating the least
energy $m_\lambda$ on the sign-changing Nehari manifold
$\mathcal{M}_\lambda$ and proving the $(\mathrm{PS})_{m_\lambda}$
condition, respectively.

\subsection{Some preliminary lemmas}

\noindent In this subsection, we will prove some preliminary
results which play an important role in the proof of Theorem
\ref{thm-2}. First we show that the sign-changing Nehari manifold
$\mathcal{M}_\lambda$ is non-empty.

\begin{lem}\label{lam-16}
Assume that $\lambda<0$ holds. Let $u\in H_r^{1}(\mathbb{R}^3)$
with $u^\pm\neq0$, then there exists a unique pair $(s_u,t_u)\in
(0,\infty)\times (0,\infty)$ such that $s_uu^++t_uu^-\in
\mathcal{M}_\lambda$. Moreover,
$$\mathcal{I }_\lambda(s_uu^++t_uu^-)
=\max\limits_{s,t\ge0}\,\mathcal{I}_\lambda(su^++tu^-).$$
\end{lem}

\begin{proof} %
Firstly, we prove the existence of $(s_u,t_u)$. Define the
function $\Phi:[0,\infty)\times[0,\infty)\to \mathbb{R}$ by
\begin{align*}
\Phi(s,t)
:&=\mathcal{I}_\lambda(su^++tu^-)
\\&=\frac{s^2}{2}\|u^+\|^2
+\frac{t^2}{2}\|u^-\|^2
+\frac{\lambda s^{10}}{10}
\int_{\mathbb{R}^{3}}\phi_{u^+}|u^+|^{5}\mathrm{d}x
+\frac{\lambda t^{10}}{10}
\int_{\mathbb{R}^{3}}\phi_{u^-}|u^-|^{5}\mathrm{d}x
\\&\quad+\frac{\lambda
s^{5}t^{5}}{10}\int_{\mathbb{R}^{3}}\phi_{u^-}
|u^+|^{5}\mathrm{d}x
+\frac{\lambda s^{5}t^{5}}{10}\int_{\mathbb{R}^{3}}\phi_{u^+}
|u^-|^{5}\mathrm{d}x
-\frac{s^{6}}{6}\int_{\mathbb{R}^{3}}|u^+|^{6}\mathrm{d}x
\\&\quad-\frac{t^{6}}{6}\int_{\mathbb{R}^{3}}|u^-|^{6}\mathrm{d}x
-\frac{s^{q}}{q}\int_{\mathbb{R}^{3}}|u^+|^{q}\mathrm{d}x
-\frac{t^{q}}{q}\int_{\mathbb{R}^{3}}|u^-|^{q}\mathrm{d}x,
\end{align*}
where $u^\pm\neq 0$. Through a simple calculation, we obtain that
$su^++tu^-\in \mathcal{M}_\lambda$ is equivalent to
$\left(\frac{\partial\Phi(s,t)}{\partial
s},\frac{\partial\Phi(s,t)}{\partial t}\right)=(0,0)$ with $s,t>
0$.

Secondly, we prove the uniqueness of the pair $(s_u,t_u)$
obtained above. By the definition of $\Phi$,
\begin{align*}
\Phi(s,t)
&\le \frac{s^2}{2}\|u^+\|^2
+\frac{t^2}{2}\|u^-\|^2
-\frac{s^{6}}{6}\int_{\mathbb{R}^{3}}|u^+|^{6}\mathrm{d}x
-\frac{t^{6}}{6}\int_{\mathbb{R}^{3}}|u^-|^{6}\mathrm{d}x,
\end{align*}
from which we obtain that
\begin{align*}
\lim\limits_{s^2+t^2\to \infty}\Phi(s,t)
&\le \lim\limits_{s^2+t^2\to \infty}
\left(\frac{s^2}{2}\|u^+\|^2
+\frac{t^2}{2}\|u^-\|^2
-\frac{s^{6}}{6}\int_{\mathbb{R}^{3}}|u^+|^{6}\mathrm{d}x
-\frac{t^{6}}{6}\int_{\mathbb{R}^{3}}|u^-|^{6}\mathrm{d}x
\right)
=-\infty.
\end{align*}
Thus, $\Phi(s,t)$ has at least one global maximum point on
$[0,\infty)\times[0,\infty)$. On the other hand, since the
function $\Phi$ is strictly concave, then any critical point is a
maximum point and there exists at most one maximum point.
Thereby, there exists a unique maximum point for the function
$\Phi$ on $[0,\infty)\times[0,\infty)$.

Finally, it is sufficient to verify that the maximum point cannot
be achieved on the boundary of $[0,\infty)\times[0,\infty)$.
Suppose by contradiction that $(\bar{s},0)$ is a maximum point of
$\Phi$ with $\bar{s}\ge 0$. If $t>0$ small enough, it holds that
\begin{align*}
\frac{\partial\Phi(\bar{s},t)}{\partial t}
&=t\|u^-\|^2
+\lambda{t^{9}}\int_{\mathbb{R}^{3}}\phi_{u^-}
|u^-|^{5}\mathrm{d}x
+\frac{\lambda\overline{s}^{5}t^{4}}{2}
\int_{\mathbb{R}^{3}}\phi_{u^-} |u^+|^{5}\mathrm{d}x
+\frac{\lambda\overline{s}^{5}t^{4}}{2}
\int_{\mathbb{R}^{3}}\phi_{u^+} |u^-|^{5}\mathrm{d}x
\\&\quad
-{t^{5}}\int_{\mathbb{R}^{3}}|u^-|^{6}\mathrm{d}x
-{t^{q-1}}\int_{\mathbb{R}^{3}}|u^-|^{q}\mathrm{d}x
>0,
\end{align*}
that is, $\Phi(\bar{s},t)$ is an increasing function with respect
to $t$ if $t>0$ small enough, which yields to a contradiction.
Similarly, $\Phi(s,t)$ cannot achieve its global maximum on
$(0,\bar{t})$. So, the proof is completed.
\end{proof}

Let $(s_u,t_u)\in(0,\infty)\times (0,\infty)$ be the unique pair
obtained by Lemma \ref{lam-16}. In the following, we will study
some further properties of $(s_u,t_u)$.

\begin{lem}\label{lam-126}
For any $u\in H_r^1(\mathbb{R}^3)$ with $u^\pm \neq 0$, there
hold

  (1) the functionals $s,t$ are continuous in
  $H_r^1(\mathbb{R}^3)$;

  (2) if $u_n^+\rightarrow0$ in $H_r^1(\mathbb{R}^3)$ and
  $u_n^-\to u^-\neq 0$ in $H_r^1(\mathbb{R}^3)$ as
  $n\rightarrow\infty$, one obtains that
  $s_{u_n}\rightarrow\infty$;
  if $u_n^-\rightarrow0$ in $H_r^1(\mathbb{R}^3)$ and $u_n^+\to
  u^+\neq 0$ in $H_r^1(\mathbb{R}^3)$ as $n\rightarrow\infty$,
  one has that $t_{u_n}\rightarrow\infty$;

  (3) if $\{u_n\}\subset\mathcal{M}_\lambda$,
  $\lim_{n\rightarrow\infty}\mathcal{I}_\lambda(u_n)=m_\lambda$,
  then $m_\lambda>0$, $\Lambda_2\le\|u_n^\pm\|\le \Lambda_1$
  for some $\Lambda_{1},\Lambda_{2}>0$.
\end{lem}

\begin{proof}
($1$) Take a sequence $\{u_n\}\subset H_r^{1}(\mathbb{R}^3)$ such
that $u_n\rightarrow u$ in $H_r^{1}(\mathbb{R}^3)$, then we get
$u_n^\pm\rightarrow u^\pm$ in $H_r^{1}(\mathbb{R}^3)$. Using
Lemma \ref{lam-16}, there exist $(s_{u_n},t_{u_n})$, $(s_u,t_u)$
such that $s_{u_n}u_n^++t_{u_n}u_n^-\in\mathcal{M}_\lambda,\
s_{u}u^++t_{u}u^-\in\mathcal{M}_\lambda$. By the definition of
$\mathcal{M}_\lambda$ and $\lambda<0$, we see that
\begin{align}
{s_{u_n}^2}\|u_n^+\|^2
&\ge{s_{u_n}^2}\|u_n^+\|^2
+\lambda s_{u_n}^{10}
\int_{\mathbb{R}^{3}}\phi_{u_n^+}|u_n^+|^{5}\mathrm{d}x
+\lambda{s_{u_n}^{5}}{t_{u_n}^{5}}
\int_{\mathbb{R}^{3}}\phi_{u_n^-}|u_n^+|^{5}\mathrm{d}x
\nonumber\\&=s_{u_n}^{6}\int_{\mathbb{R}^{3}}|u_n^+|^{6}\mathrm{d}x
+s_{u_n}^{q}\int_{\mathbb{R}^{3}}|u_n^+|^{q}\mathrm{d}x,\label{2.266}\\
{t_{u_n}^2}\|u_n^-\|^2
&\ge{t_{u_n}^2}\|u_n^-\|^2
+\lambda t_{u_n}^{10}
\int_{\mathbb{R}^{3}}\phi_{u_n^-}|u_n^-|^{5}\mathrm{d}x
+\lambda {s_{u_n}^{5}}{t_{u_n}^{5}}
\int_{\mathbb{R}^{3}}\phi_{u_n^+}|u_n^-|^{5}\mathrm{d}x
\nonumber\\&=t_{u_n}^{6}\int_{\mathbb{R}^{3}}|u_n^-|^{6}\mathrm{d}x
+t_{u_n}^{q}\int_{\mathbb{R}^{3}}|u_n^-|^{q}\mathrm{d}x.\label{2.276}
\end{align}
We claim that $\{s_{u_n}\}$ and $\{t_{u_n}\}$ are bounded in
$\mathbb{R}^+$. It yields from \eqref{2.266} that
\begin{align*}
s_{u_n}^{6}\int_{\mathbb{R}^{3}}|u_n^+|^{6}\mathrm{d}x\le
{s_{u_n}^2}\|u_n^+\|^2,
\end{align*}
which implies that $\{s_{u_n}\}$ is bounded in $\mathbb{R}^+$.
Analogously, $\{t_{u_n}\}$ is bounded in $\mathbb{R}^+$.
Therefore, up to a subsequence if necessary, still denoted by
$\{s_{u_n}\}$ and $\{t_{u_n}\}$, there exists a pair of
nonnegative numbers $(s_0,t_0)$ such that
\begin{equation*}
\lim\limits_{n\to\infty}s_{u_n}= s_0 \quad \mathrm{and} \quad
\lim\limits_{n\to\infty}t_{u_n}= t_0.
\end{equation*}
Passing to the limit in \eqref{2.266} and \eqref{2.276}, we
deduce that
\begin{align*}
{s_{0}^2}\|u^+\|^2
+\lambda
s_{0}^{10}\int_{\mathbb{R}^{3}}\phi_{u^+}|u^+|^{5}\mathrm{d}x
+\lambda{s_{0}^{5}}{t_{0}^{5}}
\int_{\mathbb{R}^{3}}\phi_{u^-}|u^+|^{5}\mathrm{d}x
&=s_{0}^{6}\int_{\mathbb{R}^{3}}|u^+|^{6}\mathrm{d}x
+s_{0}^{q}\int_{\mathbb{R}^{3}}|u^+|^{q}\mathrm{d}x,\\
{t_{0}^2}\|u^-\|^2
+\lambda
t_{0}^{10}\int_{\mathbb{R}^{3}}\phi_{u^-}|u^-|^{5}\mathrm{d}x
+\lambda{s_{0}^{5}}{t_{0}^{5}}
\int_{\mathbb{R}^{3}}\phi_{u^+}|u^-|^{5}\mathrm{d}x
&=t_{0}^{6}\int_{\mathbb{R}^{3}}|u^-|^{6}\mathrm{d}x
+t_{0}^{q}\int_{\mathbb{R}^{3}}|u^-|^{q}\mathrm{d}x,
\end{align*}
which indicates that $s_0u^++t_0u^-\in\mathcal{M}_\lambda$.
According to the uniqueness of $(s_u,t_u)$, we conclude that
$s_u=s_0$ and $t_u=t_0$.

($2$) Here we only need to show that $s_{u_n}\rightarrow\infty$
if $u_n^+\rightarrow0$ in $H_r^{1}(\mathbb{R}^3)$ and $u_n^-\to
u^-\neq 0$ in $H_r^{1}(\mathbb{R}^3)$. Analogously,
$t_{u_n}\rightarrow\infty$ if $u_n^-\rightarrow0$ in
$H_r^{1}(\mathbb{R}^3)$ and $u_n^+\to u^+\neq 0$ in
$H_r^{1}(\mathbb{R}^3)$. In fact, by contradiction, if
$u_n^+\rightarrow0$ in $H_r^{1}(\mathbb{R}^3)$ and $u_n^-\to
u^-\neq 0$ in $H_r^{1}(\mathbb{R}^3)$, there exists a constant
$M>0$ such that $s_{u_n}\le M$. By the Sobolev inequality,
Proposition \ref{prop-1}($3$) and $q\in(2,6)$, it gives that
\begin{align}
s_{u_n}^{4}\int_{\mathbb{R}^{3}}|u_n^+|^{6}\mathrm{d}x
&\le C_1 \|u_n^+\|^{6}=o(\|u_n^+\|^{2}),\label{3.16}\\
s_{u_n}^{q-2}\int_{\mathbb{R}^{3}}|u_n^+|^{q}\mathrm{d}x
&\le C_2 \|u_n^+\|^{q}=o(\|u_n^+\|^{2}),\label{3.46}\\
s_{u_n}^{8}\int_{\mathbb{R}^{3}}\phi_{u_n^+}|u_n^+|^{5}\mathrm{d}x
&\le C_3 \|u_n^+\|^{10}=o(\|u_n^+\|^{2}).\label{3.26}
\end{align}
By \eqref{2.7} and $\lambda<0$, it holds that
$$0=\frac{\langle
\mathcal{I}_\lambda'(s_{u_n}u_n^++t_{u_n}u_n^-),t_{u_n}u_n^-\rangle}
{t_{u_n}^2}
\le \|u_n^-\|^{2}
-t_{u_n}^{4}\int_{\mathbb{R}^{3}}|u_n^-|^{6}\mathrm{d}x,$$
which implies that $\{t_{u_n}\}$ is bounded in $\mathbb{R}^+$.
Then, by $u_n^-\rightarrow u^-\neq 0$ in $H_r^{1}(\mathbb{R}^3)$,
it yields that
\begin{align}\label{3.36}
s_{u_n}^{3}t_{u_n}^{5}
\int_{\mathbb{R}^{3}}\phi_{u_n^-}|u_n^+|^{5}\mathrm{d}x
&\le C_4 \|u_n^+\|^{5}\|u_n^-\|^{5}=o(\|u_n^+\|^{2}).
\end{align}
From \eqref{3.16}$-$\eqref{3.36}, we obtain that
\begin{align*}
0=\frac{\langle
\mathcal{I}_\lambda'(s_{u_n}u_n^++t_{u_n}u_n^-),s_{u_n}u_n^+\rangle}
{s_{u_n}^2}
&=\|u_n^+\|^{2}
+\lambda
s_{u_n}^{8}\int_{\mathbb{R}^{3}}\phi_{u_n^+}|u_n^+|^{5}\mathrm{d}x
+\lambda s_{u_n}^{3}t_{u_n}^{5}
\int_{\mathbb{R}^{3}}\phi_{u_n^-}|u_n^+|^{5}\mathrm{d}x
\\&\quad -s_{u_n}^{4}
\int_{\mathbb{R}^{3}}|u_n^+|^{6}\mathrm{d}x
-s_{u_n}^{q-2}\int_{\mathbb{R}^{3}}|u_n^+|^{q}\mathrm{d}x
\\&\ge \|u_n^+\|^{2}-o(\|u_n^+\|^{2})>0,
\end{align*}
which is a contradiction. Thus, if $u_n^+\rightarrow0$ in
$H_r^{1}(\mathbb{R}^3)$ and $u_n^-\to u^-\neq 0$ in
$H_r^{1}(\mathbb{R}^3)$, $s_{u_n}\rightarrow\infty$.

($3$) One hand, for any $\{{u}_n\}\subset
\mathcal{M}_\lambda\subset\mathcal{N}_\lambda$, it holds that
 $\langle \mathcal{I}_\lambda'({u}_n),{u}_n\rangle=0$, and thus
\begin{align*}
m_\lambda+o(1)
&=\mathcal{I}_\lambda({{u}_n})
=\mathcal{I}_\lambda({{u}_n})
-\frac{1}{q}\langle \mathcal{I}_\lambda '({{u}_n}),{{u}_n}\rangle
\\&=\left(\frac{1}{2}-\frac{1}{q}\right)\|{{u}_n}\|^2
+\left(\frac{\lambda}{10}-\frac{\lambda}{q}\right)
\int_{\mathbb{R}^{3}}\phi_{u_n}|{u}_n|^{5}\mathrm{d}x
+\left(\frac{1}{q}-\frac{1}{6}\right)
\int_{\mathbb{R}^{3}}|{u}_n|^{6}\mathrm{d}x
\\&\ge\left(\frac{1}{2}-\frac{1}{q}\right)\|{{u}_n}\|^2
\ge \left(\frac{1}{2}-\frac{1}{q}\right)\|u_n^\pm\|^2>0,
\end{align*}
which means that $m_\lambda>0$ and there exists $\Lambda_1>0$
such that $\|u_n^\pm\|\le \Lambda_1$. On the other hand, by
$\{{u}_n\}\subset{\mathcal{M}}_\lambda$ and the Sobolev
inequality, one derives from $\|u_n^-\|\le \Lambda_1$ that
\begin{align*}
0<\|u_n^+\|^2
&=\int_{\mathbb{R}^{3}}|u_n^+|^{6}\mathrm{d}x
+\int_{\mathbb{R}^{3}}|u_n^+|^{q}\mathrm{d}x
-\lambda\int_{\mathbb{R}^{3}}\phi_{u_n^+}|u_n^+|^{5}\mathrm{d}x
-\lambda\int_{\mathbb{R}^{3}}\phi_{u_n^-}|u_n^+|^{5}\mathrm{d}x
\\&\le C_1\|u_n^+\|^{6}
+C_2\|u_n^+\|^{q}
+C_3\|u_n^+\|^{10}
+C_4\|u_n^-\|^{5}\|u_n^+\|^{5}
\\&\le
C_1\|u_n^+\|^{6}
+C_2\|u_n^+\|^{q}
+C_3\|u_n^+\|^{10}
+C_5\|u_n^+\|^{5}.
\end{align*}
Thus, there exists $\Lambda_2>0$ such that $\|u_n^+\|\ge \Lambda_2$. Similarly, we also infer that $\|u_n^-\|\ge \Lambda_2$. Hence, $\Lambda_2\le\|{u}_n^\pm\|\le \Lambda_1$ for some $\Lambda_1,\Lambda_2>0$. This completes
the proof.
\end{proof}

Inspired by \cite{KJB03}, the following results hold. Since the
proof is standard, we omit it here.

\begin{lem}\label{lam-64}
Suppose that $\lambda<0$ and $q\in(2,6)$ hold. It holds that

(1) for any $u\in H_r^1(\mathbb{R}^3)\setminus\{0\}$, there
exists a unique $\bar{s}_u>0$ such that
$\bar{s}_uu\in\mathcal{N}_\lambda $. Moreover,
$$\mathcal{I}_\lambda ({\bar{s}_uu})
=\max\limits_{s\ge0}\,\mathcal{I} _\lambda({s u});$$

(2) system \eqref{1.6} possesses a positive least energy solution
$v_0\in\mathcal{N}_\lambda$ such that $\mathcal{I}_\lambda
(v_0)=c_{\lambda}$.
\end{lem}

\begin{rem}\label{rem-6}
Recall that $v_0$ is a positive least energy solution of system
\eqref{1.6}. We can directly deduce from \cite[Theorem
1.11]{LGB90} that $v_0\in L^\infty(\mathbb{R}^3)$ and $v\in
\mathcal{C}_{\mathrm{loc}}^{1,\alpha}(\mathbb{R}^3)$ for some
$0<\alpha<1$. The boundedness and regularity of $v_0$ play a
crucial role in the proof of Theorem \ref{thm-2}, see Lemma
\ref{lam-86} below.
\end{rem}


\subsection{Constructing the sign-changing
$(\mathrm{PS})_{m_\lambda}$ sequence}\label{13}

\noindent Due to the presence of a critical nonlinearity in
system \eqref{1.6}, our first task is to construct a
sign-changing $(\mathrm{PS})_{m_\lambda}$ sequence. Inspired by
the spirit of \cite{CG36}, we present some definitions firstly.
Let the functional $l_\lambda (u,v)$ defined on
$H_r^1(\mathbb{R}^3)$ by
\begin{equation*}
l_\lambda(u,v):=
\begin{cases}
\frac{
\int_{\mathbb{R}^{3}}|u|^{6}\mathrm{d}x
+\int_{\mathbb{R}^{3}}|u|^{q}\mathrm{d}x
-\lambda\int_{\mathbb{R}^{3}}\phi_{u}|u|^{5}\mathrm{d}x
-\lambda\int_{\mathbb{R}^{3}}\phi_{v}|u|^{5}\mathrm{d}x}
{\|u\|^2
},\ \ &\mathrm{if}\ u\neq0;\\
0,\quad\quad &\mathrm{if}\ u=0.
\end{cases}
\end{equation*}
Obviously, $l_\lambda(u,v)>0$ if $\lambda<0$ and $u\neq 0$. $u\in
\mathcal{M}_\lambda$ if and only if
$l_\lambda(u^+,u^-)=l_\lambda(u^-,u^+)=1$. Next we define
\begin{equation}\label{4.26}
U_\lambda:
=\left\{u\in
H_r^1(\mathbb{R}^3):\frac{1}{2}<l_\lambda(u^+,u^-)<\frac{3}{2},\
\frac{1}{2}<l_\lambda(u^-,u^+)<\frac{3}{2}\right\}.
\end{equation}
Let $P$ be the cone of nonnegative functions in
$H_r^{1}(\mathbb{R}^3)$ and $Q\in[0,1]\times[0,1]$. $\Sigma$
denotes the set which contains continuous maps $\sigma$
satisfying the following two conditions:

\vspace{0.125em}
\vspace*{0.125em}

{\rm(I)} $\sigma(s,0)=0$, $\sigma(0,t)\in P$ and $\sigma(1,t)\in
-P$;

\vspace{0.125em}
\vspace*{0.125em}

{\rm(II)} $(\mathcal{I}_\lambda\circ\sigma)(s,1)\le0$,
$\frac{\int_{\mathbb{R}^{3}}|\sigma(s,1)|^{6}\mathrm{d}x
+\int_{\mathbb{R}^{3}}|\sigma(s,1)|^{q}\mathrm{d}x
-\lambda\int_{\mathbb{R}^{3}}\phi_{\sigma(s,1)}|\sigma(s,1)|^{5}\mathrm{d}x}
{\|\sigma(s,1)\|^2}
\ge 2$,

\vspace{0.25em}
\vspace*{0.125em}

\noindent where $s,t\in[0,1]$. For any $u\in
H_r^{1}(\mathbb{R}^3)$ with $u^\pm\neq0$, taking
$\sigma(s,t)=kt(1-s)u^++kstu^-$, where $k>0$, $s,t\in[0,1]$.
Notice that $\sigma(s,t)\in\Sigma$ for $k>0$ large enough, which
indicates that $\Sigma\neq\emptyset$.

\begin{lem}\label{lam-76}
There exists a sequence $\{u_n\}\subset U_\lambda$ satisfying
$\mathcal{I}_\lambda(u_n)\rightarrow m_\lambda$ and
$\mathcal{I}_\lambda'(u_n)\rightarrow 0$.
\end{lem}

\begin{proof}
We will divide the proof into the following three claims.

\textbf{Claim 1}:
$\inf_{\sigma\in\Sigma}\sup_{u\in\sigma(Q)}\mathcal{I}_\lambda(u)
=\inf_{u\in{\mathcal{M}}_\lambda} \mathcal{I}_\lambda(u)
=m_\lambda$.

On the one hand, for any $u\in {\mathcal{M}_\lambda}$, there
exists $\sigma(s,t)=kt(1-s)u^++ktsu^-\in \Sigma$ for $k>0$ large
enough. Thus, it follows from Lemma \ref{lam-16} that
\begin{align*}
\mathcal{I}_\lambda(u)
=\max\limits_{s,t\ge0}\,\mathcal{I}_\lambda(su^++tu^-)
\ge\sup\limits_{u\in\sigma(Q)}\mathcal{I}_\lambda(u)
\ge\inf\limits_{\sigma\in\Sigma}\sup\limits_{u\in\sigma(Q)}
\mathcal{I}_\lambda(u),
\end{align*}
which indicates that
\begin{equation}\label{9.1}
\inf\limits_{u\in{\mathcal{M}_\lambda}}\mathcal{I}_\lambda(u)
\ge\inf\limits_{\sigma\in\Sigma}\sup\limits_{u\in\sigma(Q)}
\mathcal{I}_\lambda(u).
\end{equation}
On the other hand, for each $\sigma\in\Sigma$ and $t\in[0,1]$, we
get that $\sigma(0,t)\in P$ and $\sigma(1,t)\in -P$, thus
\begin{align}
l_\lambda(\sigma^+(0,t),\sigma^-(0,t))
-l_\lambda(\sigma^-(0,t),\sigma^+(0,t))
&=l_\lambda(\sigma^+(0,t),\sigma^-(0,t))
\ge0,\label{4.56}\\
l_\lambda(\sigma^+(1,t),\sigma^-(1,t))
-l_\lambda(\sigma^-(1,t),\sigma^+(1,t))
&=-l_\lambda(\sigma^-(1,t),\sigma^+(1,t))
\le0.\label{4.66}
\end{align}
Meanwhile, from the definition of $\Sigma$, for any
$\sigma\in\Sigma$ and $s\in[0,1]$, we deduce that
\begin{align*}
&l_\lambda(\sigma^+(s,1),\sigma^-(s,1))
+l_\lambda(\sigma^-(s,1),\sigma^+(s,1))
\\&\quad\ge \frac{
\int_{\mathbb{R}^{3}}|\sigma(s,1)|^{6}\mathrm{d}x
+\int_{\mathbb{R}^{3}}|\sigma(s,1)|^{q}\mathrm{d}x
-\lambda\int_{\mathbb{R}^{3}}\phi_{\sigma(s,1)}|\sigma(s,1)|^{5}\mathrm{d}x}
{\|\sigma(s,1)\|^2}
\ge2,
\end{align*}
with the help of  the elementary inequality
$\frac{b}{a}+\frac{d}{c}\ge\frac{b+d}{a+c}$ for any $a,b,c,d>0$.
Thus,
\begin{align}
l_\lambda(\sigma^+(s,1),\sigma^-(s,1))
+l_\lambda(\sigma^-(s,1),\sigma^+(s,1))-2
&\ge0,\label{4.76}\\
l_\lambda(\sigma^+(s,0),\sigma^-(s,0))
+l_\lambda(\sigma^-(s,0),\sigma^+(s,0))-2
&=-2<0.\label{4.86}
\end{align}
Combining Miranda's theorem \cite{MC5} with
\eqref{4.56}$-$\eqref{4.86}, there exists $(s_\sigma,t_\sigma)
\in Q$ such that
\begin{align*}
0
&=l_\lambda(\sigma^+(s_\sigma,t_\sigma),\sigma^-(s_\sigma,t_\sigma))
-l_\lambda(\sigma^-(s_\sigma,t_\sigma),\sigma^+(s_\sigma,t_\sigma))
\\&=l_\lambda(\sigma^+(s_\sigma,t_\sigma),\sigma^-(s_\sigma,t_\sigma))
+l_\lambda(\sigma^-(s_\sigma,t_\sigma),\sigma^+(s_\sigma,t_\sigma))
-2,
\end{align*}
which directly gives that
\begin{equation*}
l_\lambda(\sigma^+(s_\sigma,t_\sigma),\sigma^-(s_\sigma,t_\sigma))
=l_\lambda(\sigma^-(s_\sigma,t_\sigma),\sigma^+(s_\sigma,t_\sigma))
=1.
\end{equation*}
This indicates that for any $\sigma\in\Sigma$, there exists
$u_{\sigma}=\sigma(s_\sigma,t_\sigma)\in{\sigma(Q)}
\cap{\mathcal{M}_\lambda}$, which yields that
\begin{equation*}
\sup\limits_{u\in\sigma(Q)}\mathcal{I}_\lambda(u)
\ge \mathcal{I}_\lambda(u_\sigma)
\ge\inf\limits_{u\in{\mathcal{M}_\lambda}}\mathcal{I}_\lambda(u),
\end{equation*}
that is,
\begin{equation}\label{9.2}
\inf\limits_{\sigma\in\Sigma}\sup\limits_{u\in\sigma(Q)}
\mathcal{I}_\lambda(u)
\ge\inf\limits_{u\in{\mathcal{M}_\lambda}}\mathcal{I}_\lambda(u).
\end{equation}
Hence, combining \eqref{9.1} with \eqref{9.2}, we conclude that
\begin{equation*}
\inf\limits_{\sigma\in\Sigma}\sup\limits_{u\in\sigma(Q)}\mathcal{I}_\lambda(u)
=\inf\limits_{u\in{\mathcal{M}_\lambda}}\mathcal{I}_\lambda(u)
=m_\lambda.
\end{equation*}

\textbf{Claim 2}: There exists a $(\mathrm{PS})_{m_\lambda}$
sequence $\{u_n\}\subset H_r^{1}(\mathbb{R}^3)$ for the
functional $\mathcal{I}_\lambda$.

Consider a minimizing sequence
$\{{w_n}\}\subset{\mathcal{M}_\lambda}$ and
${\sigma}_n(s,t)=kt(1-s)w_n^++ktsw_n^-\in\Sigma$, then
\begin{align*}
\lim\limits_{n\rightarrow\infty}\max\limits_{w\in{\sigma}_n(Q)}\mathcal{I}_\lambda(w)
=\lim\limits_{n\rightarrow\infty}\mathcal{I}_\lambda({w_n}).
\end{align*}
In view of a variant form of the classical deformation lemma
\cite{PHR51} due to Hofer \cite{HH50}, we assert that there
exists $\{u_n\}\subset H_r^{1}(\mathbb{R}^3)$ such that, as
$n\rightarrow\infty$,
\begin{equation}\label{4.96}
\mathcal{I}_\lambda(u_n)\rightarrow m_\lambda,\quad
\mathcal{I}_\lambda'(u_n)\rightarrow0 \ \ \mathrm{and}\ \
\mathrm{dist}(u_n,{\sigma}_n(Q))\rightarrow0.
\end{equation}
Suppose by contradiction, there exists $\delta>0$ such that
$\sigma_n(Q)\cap V_\delta=\emptyset$ for $n$ large enough, where
\begin{equation*}
V_\delta=\left\{u\in H_r^{1}(\mathbb{R}^3):\exists v\in
H_r^{1}(\mathbb{R}^3),\ s.t.\ \|v-u\|\le \delta,\
\|\mathcal{I}_\lambda'(v)\|\le \delta,\
|\mathcal{I}_\lambda(v)-m_\lambda|\le \delta\right\}.
\end{equation*}
Inspired by \cite[Lemma 1]{HH50}, there exists a continuous map
$\eta:[0,1]\times H_r^{1}(\mathbb{R}^3) \to
H_r^{1}(\mathbb{R}^3)$ such that for some
$\epsilon\in(0,\frac{m_\lambda}{2})$ and all $t\in[0,1]$,

{$(a)$} $\eta(0,u)=u$, $\eta(t,-u)=-\eta(t,u)$;

{$(b)$} $\eta(t,u)=u$, for any $u\in
\mathcal{I}_\lambda^{m_\lambda-\epsilon}\cup
(H_r^{1}(\mathbb{R}^3)\setminus
\mathcal{I}_\lambda^{m_\lambda+\epsilon})$;

{$(c)$}
$\eta(1,\mathcal{I}_\lambda^{m_\lambda+\frac{\epsilon}{2}}\setminus
V_\delta)\subset
\mathcal{I}_\lambda^{m_\lambda-\frac{\epsilon}{2}}$;

{$(d)$}
$\eta(1,(\mathcal{I}_\lambda^{m_\lambda+\frac{\epsilon}{2}}\cap
P)\setminus V_\delta)\subset
\mathcal{I}_\lambda^{m_\lambda-\frac{\epsilon}{2}}\cap P$, where
$\mathcal{I}_\lambda^k:=\{u\in
H_r^{1}(\mathbb{R}^3):\mathcal{I}_\lambda(u)\le k\}$.

Since $\lim_{n\rightarrow\infty}\max_{w\in{\sigma}_n(Q)}
\mathcal{I}_\lambda(w)
=\lim_{n\rightarrow\infty}\mathcal{I}_\lambda(w_n)=m_\lambda$,
choosing $n$ large enough such that
\begin{equation}\label{556}
\sigma_n(Q)\subset
\mathcal{I}_\lambda^{m_\lambda+\frac{\epsilon}{2}}\quad
\mathrm{and} \quad \sigma_n(Q)\cap V_\delta=\emptyset.
\end{equation}
Denote $\tilde{\sigma}_n(s,t):=\eta(1,\sigma_n(s,t))$ for all
$(s,t)\in Q$. We declare that $\tilde{\sigma}_n\in \Sigma$, it
derives from \eqref{556} and property $(c)$ of $\eta$ that
$\tilde{\sigma}_n(Q)\subset
\mathcal{I}_\lambda^{m_\lambda-\frac{\epsilon}{2}}$, which leads
to a contradiction with
\begin{align*}
m_\lambda
=\inf\limits_{\sigma\in\Sigma}\sup\limits_{w\in\sigma(Q)}
\mathcal{I}_\lambda(w)
\le \max\limits_{w\in\tilde{\sigma}_n(Q)}\mathcal{I}_\lambda(w)
\le m_\lambda-\frac{\epsilon}{2}.
\end{align*}
Actually, property $(b)$ of $\eta$ and $\sigma_n\in \Sigma$ imply
that $\tilde{\sigma}_n(s,0)=\eta(1,\sigma_n(s,0))=\eta(1,0)=0$.
One hand, it follows from $\sigma_n(0,t)\in P$, \eqref{556} and
property $(d)$ of $\eta$ that $\tilde{\sigma}_n(0,t)\in P$. On
the other hand, thanks to $\sigma_n(1,t)\in -P$ and \eqref{556},
we deduce that $-\sigma_n(1,t)\in
(\mathcal{I}_\lambda^{m_\lambda+\frac{\epsilon}{2}}\cap
P)\setminus V_\delta $, which implies that
$\tilde{\sigma}_n(1,t)=\eta(1,\sigma_n(1,t))=-\eta(1,-\sigma_n(1,t))\in
-P$ with the help of properties $(a)$ and $(d)$ of $\eta$. Then,
$\tilde{\sigma}_n$ satisfies property $(\mathrm{I})$. Moreover,
using the fact that $(\mathcal{I}_\lambda\circ \sigma_n)(s,1)\le
0$ and property $(b)$ of $\eta$, we obtain that
$\tilde{\sigma}_n(s,1)=\eta(1,\sigma_n(s,1))=\sigma_n(s,1)$, then
$\tilde{\sigma}_n$ satisfies property $(\mathrm{II})$. Hence, due
to the continuity of $\eta$ and $\sigma_n$, we derive that
$\tilde{\sigma}_n\in \Sigma.$

\textbf{Claim 3}: The sequence $\{u_n\}$ obtained in Claim 2
satisfies $\{u_n\}\subset U_\lambda$ for $n$ large enough.

Since $\mathcal{I}_\lambda'(u_n)\rightarrow0$, we obtain that
$\langle \mathcal{I }_\lambda'(u_n),u_n^\pm\rangle=o(1)$. Hence,
to complete the proof of Claim 3, it suffices to check that
$u_n^\pm\neq0$, which means that
$l_\lambda(u_n^+,u_n^-)\rightarrow1$,
$l_\lambda(u_n^-,u_n^+)\rightarrow1$, and then $\{u_n\}\subset
U_\lambda$ for $n$ large enough. From \eqref{4.96}, there exists
a sequence $\{\nu_n\}$ such that
\begin{equation}\label{4.106}
\nu_n=s_nw_n^++t_nw_n^-\in {\sigma}_n(Q),\quad
\|\nu_n-u_n\|\rightarrow0.
\end{equation}
Thus, to prove $u_n^\pm\neq0$ is equivalent to prove
$s_nw_n^+\neq0$ and $t_nw_n^-\neq0$ for $n$ large enough. By
$\{w_n\}\subset \mathcal{M}_\lambda$ and Lemma
\ref{lam-126}($3$), we just need to check that $s_n\nrightarrow0$
and $t_n\nrightarrow0$ for $n$ large enough. Suppose by
contradiction that $s_n\rightarrow0$, by the continuity of
$\mathcal{I}_\lambda$ and \eqref{4.106}, we see that
\begin{equation*}
0<m_\lambda=\lim\limits_{n\rightarrow\infty}
\mathcal{I}_\lambda(\nu_n)
=\lim\limits_{n\rightarrow\infty}
\mathcal{I}_\lambda (s_nw_n^++t_nw_n^-)
=\lim\limits_{n\rightarrow\infty}
\mathcal{I}_\lambda(t_nw_n^-).
\end{equation*}
It derives from Lemma \ref{lam-126}($3$) that
$\Lambda_2\le\|w_{n}^-\|\le \Lambda_1$, which indicates that
$t_n\nrightarrow0$ and $\{t_n\}$ is bounded. Thereby, by
$\lambda<0$, $q\in(2,6)$, $\Lambda_2\le \|w_{n}^+\|\le
\Lambda_1$ and Lemma \ref{lam-16}, one concludes that
small enough,
\begin{align*}
m_\lambda&=\lim\limits_{n\rightarrow\infty}\mathcal{I}_\lambda(w_n)
=\lim\limits_{n\rightarrow\infty}\max\limits_{s,t\ge0}\,
\mathcal{I}_\lambda(sw_n^++tw_n^-)
\\&\ge\lim\limits_{n\rightarrow\infty}\max\limits_{s\ge0}\,
\mathcal{I}_\lambda(sw_n^++t_nw_n^-)
\\&\ge\lim\limits_{n\rightarrow\infty}\max\limits_{s\ge0}
\bigg(\frac{s^2}{2}\|w_n^+\|^2
+\frac{\lambda s^{10}}{10}
\int_{\mathbb{R}^{3}}\phi_{w_n^+}|w_n^+|^{5}\mathrm{d}x
+ \frac{\lambda s^{5}t_n^{5}}{5}
\int_{\mathbb{R}^{3}}\phi_{w_n^-}|w_n^+|^{5}\mathrm{d}x
\\&\quad
-\frac{s^{6}}{{6}}\int_{\mathbb{R}^{3}}{|w_n^+|^{6}}\mathrm{d}x
-\frac{s^{q}}{{q}}\int_{\mathbb{R}^{3}}{|w_n^+|^{q}}\mathrm{d}x
\bigg)
+\lim\limits_{n\rightarrow\infty}\mathcal{I}_\lambda(t_nw_n^-)
\\&\ge\lim\limits_{n\rightarrow\infty}\max\limits_{s\ge0}\left
[{\frac{s^2}{2}}\|w_n^+\|^2
-C_1\left(s^{10}\|w_n^+\|^{10}
+s^{5}\|w_n^+\|^{5}
+s^{6}\|w_n^+\|^{6}
+s^{q}\|w_n^+\|^{q}\right)
\right]+m_\lambda
\\&\ge\max\limits_{s\ge0}\left[C_2
s^2-C_{3}\left(s^{10}+s^{5}+s^{6}+s^{q}\right)\right]
+m_\lambda
\\&>m_\lambda,
\end{align*}
which leads to a contradiction. Then $\{u_n\}\subset U_\lambda$
for $n$ large enough.
\end{proof}

\subsection{Estimating the least energy $m_\lambda$}\label{sub-5.3}

\noindent In this subsection, we are devoted to estimating the
least energy $m_\lambda$ on the sign-changing Nehari manifold
$\mathcal{M}_\lambda$. For this purpose, let us denote
$u_\varepsilon:=\varphi\circ \xi_{\varepsilon}$, where
$\xi_{\varepsilon}$ is defined by \eqref{2.4} and $\varphi$ is a
cut-off function satisfying $0\le\varphi\le1,\
\varphi|_{B_{r_0}(0)}\equiv1$ and
$\mathrm{supp}(\varphi)\subset{B_{2r_0}(0)}$ for some $r_0>0$.
Similar arguments as \cite[Lemma 1.1]{BH34}, we have the
following estimates
\begin{equation}\label{3.766}
\int_{\mathbb{R}^{3}}|\nabla u_\varepsilon|^2\mathrm{d}x
=S^{\frac{3}{2}}+O(\varepsilon^{\frac{1}{2}}),\quad \quad
\left(\int_{\mathbb{R}^{3}}|u_\varepsilon|^{6}\mathrm{d}x\right)^{\frac{1}{3}}
=S^{\frac{1}{2}}+O(\varepsilon^{\frac{3}{2}})
\end{equation}
and
\begin{equation}\label{3.86}
 \int_{\mathbb{R}^{3}}|u_\varepsilon|^r\mathrm{d}x=
 \begin{cases}
 O(\varepsilon^{\frac{r}{4}}),~~~~~~~~~~~~~&r\in[2,3);\\
 O(\varepsilon^{\frac{r}{4}}|\mathrm{ln\,
 \varepsilon}|),~~~~~~~&r=3;\\
 O(\varepsilon^{\frac{6-r}{4}}),~~~~~~~~~~~&r\in(3,6).
 \end{cases}
 \end{equation}

\begin{lem}\label{lam-gt6}
Let $A, B, C>0$ and define $g: [0,\infty)\to \mathbb{R}$ by
\begin{align*}
g(t)=\frac{A}{2}t^2+\frac{\lambda B}{10}t^{10}-\frac{C}{6}t^{6},
\end{align*}
where $\lambda<0$ is a constant. Then
\begin{equation*}
\max_{t\geq 0}g(t)=\left(\frac{C-\sqrt{C^2-4\lambda AB}}{2\lambda
B}\right)^{\frac{1}{2}}
\frac{12 \lambda AB-C^2+C\sqrt{C^2-4\lambda AB}}{30\lambda B}.
\end{equation*}
\end{lem}

\begin{proof}
For $t\geq 0$, we have $g'(t)=t\left(A+\lambda Bt^{8} -Ct^{4}
\right)$. Let $h(t)=A+\lambda Bt^{8} -Ct^{4}$, we arrive at
\begin{equation*}
t^{4}=\frac{C-\sqrt{C^2-4\lambda AB}}{2\lambda B}.
\end{equation*}
Substituting it into $g(t)$, we obtain the result, and the proof
is completed.
\end{proof}

\begin{lem}\label{lam-416}
Let $f(t)
=\frac{t^2}{2}\int_{\mathbb{R}^{3}}|\nabla
u_{\varepsilon}|^{2}\mathrm{d}x
+\frac{\lambda
t^{10}}{10}\int_{\mathbb{R}^{3}}\phi_{u_{\varepsilon}}
|u_{\varepsilon}|^{5}\mathrm{d}x
-\frac{t^{6}}{6}
\int_{\mathbb{R}^{3}}|u_{\varepsilon}|^{6}\mathrm{d}x$, for $t\ge
0$ and $\lambda<0$. Then we obtain that as $\varepsilon\to 0$,
\begin{align*}
\max\limits_{t\ge 0}f(t)
\le \left(\frac{1-\sqrt{1-4\lambda}}{2\lambda
}\right)^{\frac{1}{2}}
\frac{12 \lambda-1+\sqrt{1-4\lambda }}{30\lambda }
S^{\frac{3}{2}}+O(\varepsilon^{\frac{1}{2}}).
\end{align*}
\end{lem}

\begin{proof}
From $-\Delta \phi_{u_\varepsilon}=|u_\varepsilon|^5$, it holds
that
\begin{align*}
\int_{\mathbb{R}^3}|u_\varepsilon|^6\mathrm{d}x
=\int_{\mathbb{R}^3}-\Delta \phi_{u_\varepsilon}
|u_\varepsilon|\mathrm{d}x
&\le \frac{1}{2}\int_{\mathbb{R}^3}|\nabla
\phi_{u_\varepsilon}|^2\mathrm{d}x
+\frac{1}{2}\int_{\mathbb{R}^3}|\nabla
|{u_\varepsilon}||^2\mathrm{d}x
\nonumber\\&\le \frac{1}{2}
\int_{\mathbb{R}^3}\phi_{u_\varepsilon}|u_\varepsilon|^5\mathrm{d}x
+\frac{1}{2}\int_{\mathbb{R}^3}|\nabla
{u_\varepsilon}|^2\mathrm{d}x.
\end{align*}
For $\varepsilon>0$ small enough, it derives from the above
inequality and \eqref{3.766} that
\begin{align*}
\int_{\mathbb{R}^3}\phi_{u_\varepsilon}|u_\varepsilon|^5\mathrm{d}x
\ge 2\int_{\mathbb{R}^3}|u_\varepsilon|^6\mathrm{d}x
-\int_{\mathbb{R}^3}|\nabla {u_\varepsilon}|^2\mathrm{d}x
=S^{\frac{3}{2}}+O(\varepsilon^{\frac{1}{2}}),
\end{align*}
which together with Lemma \ref{lam-gt6} and \eqref{3.766} yields
that
\begin{align*}
f(t)
&=\frac{t^2}{2}\int_{\mathbb{R}^{3}}|\nabla
u_{\varepsilon}|^{2}\mathrm{d}x
+\frac{\lambda
t^{10}}{10}\int_{\mathbb{R}^{3}}\phi_{u_{\varepsilon}}
|u_{\varepsilon}|^{5}\mathrm{d}x
-\frac{t^{6}}{6}\int_{\mathbb{R}^{3}}|u_{\varepsilon}|^{6}\mathrm{d}x
\\&\le
\frac{t^2}{2}\left(S^{\frac{3}{2}}+O(\varepsilon^{\frac{1}{2}})\right)
+\frac{\lambda
t^{10}}{10}\left(S^{\frac{3}{2}}+O(\varepsilon^{\frac{1}{2}})\right)
-\frac{t^{6}}{6}\left(S^{\frac{3}{2}}+O(\varepsilon^{\frac{3}{2}})\right)
\\&\le \left(\frac{1-\sqrt{1-4\lambda}}{2\lambda
}\right)^{\frac{1}{2}}
\frac{12 \lambda-1+\sqrt{1-4\lambda }}{30\lambda
}S^{\frac{3}{2}}+O(\varepsilon^{\frac{1}{2}}),
\end{align*}
for $\varepsilon>0$ small enough and $\lambda<0$. So, the proof
is completed.
\end{proof}

For the simplicity of notation, let us denote
$c^*_\lambda:=\left(\frac{1-\sqrt{1-4\lambda}}{2\lambda
}\right)^{\frac{1}{2}}
\frac{12 \lambda-1+\sqrt{1-4\lambda }}{30\lambda
}S^{\frac{3}{2}}$.  Next, we will estimate the upper bound of the
least energy $m_\lambda$ on the sign-changing Nehari manifold
$\mathcal{M}_\lambda$ by using a test function, which is a key
point in this paper.

\begin{lem}\label{lam-86}
Assume that $\lambda<0$ and $q\in(5,6)$ hold. Then
$m_\lambda<c_\lambda+c^*_\lambda$, where $m_\lambda$ and $c_\lambda$ defined by \eqref{9.26} and \eqref{9.25} respectively.
\end{lem}

\begin{proof}
The main idea of this lemma is to look for an element in
$\mathcal{M}_\lambda$ such that the energy value of this element
is strictly less than $c_\lambda+c^*_\lambda$. We break the proof
into two parts.

Firstly, we assert that there exist
$s_\varepsilon,t_\varepsilon>0$ such that $s_\varepsilon
v_0-t_\varepsilon u_\varepsilon\in\mathcal{M}_\lambda$, where
$v_0$ is a positive least energy solution of system \eqref{1.6}
obtained by Lemma \ref{lam-64}. From Remark \ref{rem-6}, we
obtain that $v_0\in L^{\infty}({\mathbb{R}^{3}})$. Let us define
$\psi(r):=\frac{1}{r}v_0-u_\varepsilon$ with $r>0$, define
$r_1=\mathrm{sup}\{r\in\mathbb{R}^+:\psi^+(r)\neq0\}$ and
$r_2=\mathrm{inf}\{r\in\mathbb{R}^+:\psi^-(r)\neq0\}$. Because of
the positivity and regularity of $v_0$, it is easy to verify that
$r_1=\infty$ and $0<r_2<r_1$. As $r\rightarrow r_2^+$, this
immediately implies that $\psi^-(r)\rightarrow0$ and
$\psi^+(r)\to \frac{1}{r_2}v_0-u_\varepsilon\neq 0$. Then, we
deduce from Lemma \ref{lam-126}($2$) that
$t(\psi(r))\rightarrow\infty$. Similar to the proof of Lemma
\ref{lam-126}($1$), we see that $\left\{s(\psi(r))\right\}$ is
bounded in $\mathbb{R}^+$, and thus as $r\rightarrow r_2^+$,
\begin{equation*}
s(\psi(r))-t(\psi(r))\rightarrow-\infty.
\end{equation*}
As $r\rightarrow r_1=\infty$, $\psi^+(r)\rightarrow0$, it follows
from Lemma \ref{lam-126}$(2)$ and the proof of Lemma
\ref{lam-126}$(1)$ that $s(\psi(r))\rightarrow\infty$  and
$\left\{t(\psi(r))\right\}$ is bounded in $\mathbb{R}^+$, and
thus
\begin{equation*}
s(\psi(r))-t(\psi(r))\rightarrow\infty.
\end{equation*}
Then, Lemma \ref{lam-126}($1$) implies that there exists
$r_\varepsilon\in(r_2,r_1)$ such that
$s(\psi(r_\varepsilon))=t(\psi(r_\varepsilon))$. Let us denote
$s_\varepsilon=\frac{1}{r_\varepsilon}s(\psi(r_\varepsilon))$ and
$t_\varepsilon=t(\psi(r_\varepsilon))$, it is easy to obtain that
 \begin{align*}
s(\psi(r_\varepsilon))\psi(r_\varepsilon)
=s(\psi(r_\varepsilon))\psi^+(r_\varepsilon)
+t(\psi(r_\varepsilon))\psi^-(r_\varepsilon)
=s_\varepsilon v_0-t_\varepsilon u_\varepsilon
\in\mathcal{M}_\lambda.
\end{align*}
Moreover, it follows from Lemma \ref{lam-16} that
$\mathcal{I }_\lambda(s_\varepsilon v_0-t_\varepsilon
u_\varepsilon)=\max_{s,t\ge0}\mathcal{I}_\lambda(sv_0-tu_\varepsilon)$.

Secondly, we show that $m_\lambda<c_\lambda+c^*_\lambda$. It is
easy to check that $\mathcal{I}_\lambda(s_\varepsilon
v_0-t_\varepsilon u_\varepsilon)<0$ if $s_\varepsilon$ or
$t_\varepsilon$ large enough. Additionally, the continuity of
$\mathcal{I}_\lambda$ with respect to $t$ implies that
$\mathcal{I}_\lambda(sv_0-tu_\varepsilon)<c_\lambda+c^*_\lambda$
if $t$ small enough. Thus, it suffices to consider the case that
$s_\varepsilon,t_\varepsilon$ contained in a bounded interval.
Through a simple calculation, we can obtain that
\begin{align}\label{9.5}
m_\lambda
\le\mathcal{I}_\lambda(s_\varepsilon v_0-t_\varepsilon
u_\varepsilon)
=\mathcal{I}_\lambda(s_\varepsilon
v_0)+\Pi_1+\Pi_2+\Pi_3+\Pi_4+\Pi_5+\Pi_6,
\end{align}
where
\begin{align*}
\Pi_1&=\frac{1}{2}\|t_\varepsilon
u_\varepsilon\|_{\mathcal{D}^{1,2}(\mathbb{R}^3)}^2
+\frac{\lambda}{10}\int_{\mathbb{R}^{3}}\phi_{t_\varepsilon
u_\varepsilon}|t_\varepsilon u_\varepsilon|^{5}\mathrm{d}x
-\frac{1}{6}\int_{\mathbb{R}^{3}}|t_\varepsilon
u_\varepsilon|^{6}\mathrm{d}x,\\
\Pi_2&=\frac{1}{2}\|s_\varepsilon v_0-t_\varepsilon
u_\varepsilon\|_{\mathcal{D}^{1,2}(\mathbb{R}^3)}^2
-\frac{1}{2}\|s_\varepsilon
v_0\|_{\mathcal{D}^{1,2}(\mathbb{R}^3)}^2
-\frac{1}{2}\|t_\varepsilon
u_\varepsilon\|_{\mathcal{D}^{1,2}(\mathbb{R}^3)}^2,\\
\Pi_3&=\frac{1}{2}\int_{\mathbb{R}^{3}}\left(|s_\varepsilon
v_0-t_\varepsilon u_\varepsilon|^{2}
-|s_\varepsilon v_0|^{2}\right)\mathrm{d}x
-\frac{1}{q}\int_{\mathbb{R}^{3}}|t_\varepsilon
u_\varepsilon|^{q}
\mathrm{d}x,\\
\Pi_4&=\frac{1}{q}\int_{\mathbb{R}^{3}}\left(|s_\varepsilon
v_0|^{q}
+|t_\varepsilon u_\varepsilon|^{q}
-|s_\varepsilon v_0-t_\varepsilon
u_\varepsilon|^{q}\right)\mathrm{d}x,\\
\Pi_5&=\frac{1}{6}\int_{\mathbb{R}^{3}}\left(|s_\varepsilon
v_0|^{6}+|t_\varepsilon u_\varepsilon|^{6}-|s_\varepsilon
v_0-t_\varepsilon u_\varepsilon|^{6}\right)\mathrm{d}x,\\
\Pi_6&=\frac{\lambda}{10}\int_{\mathbb{R}^{3}}\left(\phi_{s_\varepsilon
v_0-t_\varepsilon u_\varepsilon}|s_\varepsilon v_0-t_\varepsilon
u_\varepsilon|^{5}
-\phi_{s_\varepsilon v_0}|s_\varepsilon v_0|^{5}
-\phi_{t_\varepsilon u_\varepsilon}|t_\varepsilon
u_\varepsilon|^{5}\right)\mathrm{d}x.
\end{align*}
By Lemma \ref{lam-416}, it is easy to prove that as
$\varepsilon\rightarrow0$,
\begin{equation}\label{4.116}
\Pi_1\le\max\limits_{t\ge 0}
\left(\frac{t^2}{2}\|
u_\varepsilon\|_{\mathcal{D}^{1,2}(\mathbb{R}^3)}^2
+\frac{\lambda t^{10}}{10}
\int_{\mathbb{R}^{3}}\phi_{ u_\varepsilon}|
u_\varepsilon|^{5}\mathrm{d}x
-\frac{t^6}{6}
\int_{\mathbb{R}^{3}}| u_\varepsilon|^{6}\mathrm{d}x\right)
=c^*_\lambda+O(\varepsilon^{\frac{1}{2}}).
\end{equation}
Furthermore, since $|a-b|^2\le |a|^2+|b|^2$ for all
$a,b\in\mathbb{R}^+$, by a simple calculation, we arrive at
\begin{align}\label{3.66}
\Pi_2&=\frac{1}{2}\|s_\varepsilon v_0-t_\varepsilon
u_\varepsilon\|_{\mathcal{D}^{1,2}(\mathbb{R}^3)}^2
-\frac{1}{2}\|s_\varepsilon
v_0\|_{\mathcal{D}^{1,2}(\mathbb{R}^3)}^2
-\frac{1}{2}\|t_\varepsilon
u_\varepsilon\|_{\mathcal{D}^{1,2}(\mathbb{R}^3)}^2\le 0,
\end{align}
and by \eqref{3.86}, it follows from $q\in(5,6)$ that
\begin{align}\label{3.76}
\Pi_3
&\le \frac{1}{2}\int_{\mathbb{R}^{3}}|t_\varepsilon
u_\varepsilon|^{2}
\mathrm{d}x
-\frac{1}{q}\int_{\mathbb{R}^{3}}|t_\varepsilon
u_\varepsilon|^{q}
\mathrm{d}x
= C_1 |u_\varepsilon|_2^2-C_2 |u_\varepsilon|_q^q
= C_1 \varepsilon^{\frac{1}{2}}
-C_2 \varepsilon^{\frac{6-q}{4}}.
\end{align}
To proceed further, we need the following inequality:
$|a-b|^r-a^r-b^r\ge -C(a^{r-1}b+ab^{r-1})$ for all $a,b\ge 0$ and
$r\ge 1$
 (see \cite[Calculus Lemma]{GT92}). This together with
 $q\in(5,6)$, $v_0\in L^\infty(\mathbb{R}^3)$, H\"{o}lder
 inequality and the boundedness of $s_\varepsilon,t_\varepsilon$,
 it holds that
\begin{align}\label{3.136}
\Pi_4&\le C_1\int_{\mathbb{R}^{3}}\left(|s_\varepsilon
v_0|^{q-1}|t_\varepsilon u_\varepsilon|
+|s_\varepsilon v_0||t_\varepsilon
u_\varepsilon|^{q-1}\right)\mathrm{d}x
\nonumber\\&\le C_2|v_0|_\infty^{q-1}\int_{|x|\le
2r_0}|u_\varepsilon|\mathrm{d}x
+C_3|v_0|_\infty\int_{\mathbb{R}^3}|u_\varepsilon|^{q-1}\mathrm{d}x
\nonumber\\&\le C_4\left(\int_{|x|\le
2r_0}|u_\varepsilon|^2\mathrm{d}x\right)^{\frac{1}{2}}
+C_5 \varepsilon^{\frac{7-q}{4}}
\nonumber\\&\le C_6 \varepsilon^{\frac{1}{4}}
+C_5 \varepsilon^{\frac{7-q}{4}}
\le C_7 \varepsilon^{\frac{1}{4}},
\end{align}
and
\begin{align}\label{3.146}
\Pi_5&=C_1\int_{\mathbb{R}^{3}}\left(|s_\varepsilon
v_0|^{5}|t_\varepsilon u_\varepsilon|
+|s_\varepsilon v_0||t_\varepsilon
u_\varepsilon|^{5}\right)\mathrm{d}x
\nonumber\\&\le C_2|v_0|_\infty^{5}
\int_{|x|\le 2r_0}|u_\varepsilon|\mathrm{d}x
+ C_3|v_0|_\infty
\int_{\mathbb{R}^{3}}|u_\varepsilon|^{5}\mathrm{d}x
\nonumber\\&\le C_4
\left(\int_{|x|\le
2r_0}|u_\varepsilon|^2\mathrm{d}x\right)^{\frac{1}{2}}
+C_5 \varepsilon^{\frac{1}{4}}
\nonumber\\&\le C_6 \varepsilon^{\frac{1}{4}}.
\end{align}
In view of this, it suffices to prove $\Pi_6$. Recall that
$|a-b|^r-a^r-b^r\ge -C(a^{r-1}b+ab^{r-1})$ for all $a,b\ge 0$ and
$r\ge 1$, it holds that
\begin{align}\label{9.4}
&\int_{\mathbb{R}^{3}}\phi_{s_\varepsilon v_0-t_\varepsilon
u_\varepsilon}|s_\varepsilon v_0-t_\varepsilon
u_\varepsilon|^{5}\mathrm{d}x
\nonumber\\&=
\int_{\mathbb{R}^3}\int_{\mathbb{R}^3}\frac{
|(s_\varepsilon v_0-t_\varepsilon u_\varepsilon)(y)|^{5}
|(s_\varepsilon v_0-t_\varepsilon
u_\varepsilon)(x)|^{5}}{|x-y|}\mathrm{d}x \mathrm{d}y
\nonumber\\&\ge
\int_{\mathbb{R}^3}\int_{\mathbb{R}^3}\frac{
|s_\varepsilon v_0(y)|^{5}
|s_\varepsilon v_0(x)-t_\varepsilon u_\varepsilon(x)|^{5}}
{|x-y|}\mathrm{d}x \mathrm{d}y
+\int_{\mathbb{R}^3}\int_{\mathbb{R}^3}\frac{
|t_\varepsilon u_\varepsilon(y)|^{5}
|s_\varepsilon v_0(x)-t_\varepsilon u_\varepsilon(x)|^{5}}
{|x-y|}
\mathrm{d}x \mathrm{d}y
\nonumber\\&\ -C_1\int_{\mathbb{R}^3}
\frac{
\left(|s_\varepsilon v_0(y)|^{4}|t_\varepsilon u_\varepsilon(y)|
+|s_\varepsilon v_0(y)||t_\varepsilon
u_\varepsilon(y)|^{4}\right)
|s_\varepsilon v_0(x)-t_\varepsilon u_\varepsilon(x)|^{5}}
{|x-y|}\mathrm{d}x \mathrm{d}y
\nonumber\\&\ge
\int_{\mathbb{R}^3}\int_{\mathbb{R}^3}\frac{
|s_\varepsilon v_0(y)|^{5}
\left[|s_\varepsilon v_0(x)|^{5}
\!+\!|t_\varepsilon u_\varepsilon(x)|^{5}
\!-\!C_2\left(|s_\varepsilon v_0(x)|^{4}|t_\varepsilon
u_\varepsilon(x)|
\!+\!|s_\varepsilon v_0(x)||t_\varepsilon
u_\varepsilon(x)|^{4}\right)\right]}
{|x-y|}\mathrm{d}x \mathrm{d}y
\nonumber\\&\quad+\int_{\mathbb{R}^3}\int_{\mathbb{R}^3}\frac{
|t_\varepsilon u_\varepsilon(y)|^{5}
\left[|s_\varepsilon v_0(x)|^{5}\!+\!|t_\varepsilon
u_\varepsilon(x)|^{5}\!-\!C_2\left(|s_\varepsilon
v_0(x)|^{4}|t_\varepsilon u_\varepsilon(x)|\!+\!|s_\varepsilon
v_0(x)||t_\varepsilon
u_\varepsilon(x)|^{4}\right)\right]}{|x-y|}\mathrm{d}x
\mathrm{d}y
\nonumber\\&\quad-C_1\int_{\mathbb{R}^3}\int_{\mathbb{R}^3}\frac{
\left(|s_\varepsilon v_0(y)|^{4}|t_\varepsilon
u_\varepsilon(y)|+|s_\varepsilon v_0(y)||t_\varepsilon
u_\varepsilon(y)|^{4}\right)
|s_\varepsilon v_0(x)-t_\varepsilon u_\varepsilon(x)|^{5}}
{|x-y|}\mathrm{d}x \mathrm{d}y
\nonumber\\&\ge
\int_{\mathbb{R}^{3}}\phi_{s_\varepsilon v_0}|s_\varepsilon
v_0|^{5}\mathrm{d}x
+\int_{\mathbb{R}^{3}}\phi_{t_\varepsilon
u_\varepsilon}|t_\varepsilon u_\varepsilon|^{5}\mathrm{d}x
\nonumber\\&\
-C_2\int_{\mathbb{R}^3}\int_{\mathbb{R}^3}\frac{|s_\varepsilon
v_0(y)|^{5}
\left(|s_\varepsilon v_0(x)|^{4}|t_\varepsilon
u_\varepsilon(x)|+|s_\varepsilon v_0(x)||t_\varepsilon
u_\varepsilon(x)|^{4}\right)}{|x-y|}\mathrm{d}x \mathrm{d}y
\nonumber\\&\
-C_2\int_{\mathbb{R}^3}\int_{\mathbb{R}^3}\frac{|t_\varepsilon
u_\varepsilon(y)|^{5}
\left(|s_\varepsilon v_0(x)|^{4}|t_\varepsilon
u_\varepsilon(x)|+|s_\varepsilon v_0(x)||t_\varepsilon
u_\varepsilon(x)|^{4}\right)}{|x-y|}\mathrm{d}x \mathrm{d}y
\nonumber\\&\
-C_1\int_{\mathbb{R}^3}\int_{\mathbb{R}^3}\frac{
\left(|s_\varepsilon v_0(y)|^{4}|t_\varepsilon
u_\varepsilon(y)|+|s_\varepsilon v_0(y)||t_\varepsilon
u_\varepsilon(y)|^{4}\right)
|s_\varepsilon v_0(x)-t_\varepsilon
u_\varepsilon(x)|^{5}}{|x-y|}\mathrm{d}x \mathrm{d}y.
\end{align}
With the help of the following inequality $|a-b|^r\le
2^{r-1}(|a|^r+|b|^r)$, \eqref{9.4} turns into
\begin{align}\label{9.3}
&\int_{\mathbb{R}^{3}}\phi_{s_\varepsilon v_0-t_\varepsilon
u_\varepsilon}|s_\varepsilon v_0-t_\varepsilon
u_\varepsilon|^{5}\mathrm{d}x
\nonumber\\&\ge
\int_{\mathbb{R}^{3}}\phi_{s_\varepsilon v_0}|s_\varepsilon
v_0|^{5}\mathrm{d}x
+\int_{\mathbb{R}^{3}}\phi_{t_\varepsilon
u_\varepsilon}|t_\varepsilon u_\varepsilon|^{5}\mathrm{d}x
\nonumber\\&\quad
-C_3\int_{\mathbb{R}^3}\int_{\mathbb{R}^3}\frac{|s_\varepsilon
v_0(y)|^{5}
\left(|s_\varepsilon v_0(x)|^{4}|t_\varepsilon
u_\varepsilon(x)|+|s_\varepsilon v_0(x)||t_\varepsilon
u_\varepsilon(x)|^{4}\right)}{|x-y|}\mathrm{d}x \mathrm{d}y
\nonumber\\&\quad
-C_4\int_{\mathbb{R}^3}\int_{\mathbb{R}^3}\frac{|t_\varepsilon
u_\varepsilon(y)|^{5}
\left(|s_\varepsilon v_0(x)|^{4}|t_\varepsilon
u_\varepsilon(x)|+|s_\varepsilon v_0(x)||t_\varepsilon
u_\varepsilon(x)|^{4}\right)}{|x-y|}\mathrm{d}x \mathrm{d}y.
\end{align}
According to \eqref{3.766}, \eqref{3.86}, \eqref{9.3},
Hardy-Littlewood-Sobolev inequality (see Proposition
\ref{prohlsi}) and the boundedness of
$s_\varepsilon,t_\varepsilon$, we conclude that as
$\varepsilon\to 0$,
\begin{align}\label{3.156}
\Pi_6&=\frac{\lambda}{10}\int_{\mathbb{R}^{3}}
\left(\phi_{s_\varepsilon v_0-t_\varepsilon
u_\varepsilon}|s_\varepsilon v_0-t_\varepsilon u_\varepsilon|^{5}
-\phi_{s_\varepsilon v_0}|s_\varepsilon v_0|^{5}
-\phi_{t_\varepsilon u_\varepsilon}|t_\varepsilon
u_\varepsilon|^{5}\right)\mathrm{d}x
\nonumber\\&\quad+
C_1\int_{\mathbb{R}^3}\int_{\mathbb{R}^3}\frac{|s_\varepsilon
v_0(y)|^{5}
\left(|s_\varepsilon v_0(x)|^{4}|t_\varepsilon
u_\varepsilon(x)|+|s_\varepsilon v_0(x)||t_\varepsilon
u_\varepsilon(x)|^{4}\right)}{|x-y|}\mathrm{d}x \mathrm{d}y
\nonumber\\&\quad
+C_2\int_{\mathbb{R}^3}\int_{\mathbb{R}^3}\frac{|t_\varepsilon
u_\varepsilon(y)|^{5}
\left(|s_\varepsilon v_0(x)|^{4}|t_\varepsilon
u_\varepsilon(x)|+|s_\varepsilon v_0(x)||t_\varepsilon
u_\varepsilon(x)|^{4}\right)}{|x-y|}\mathrm{d}x \mathrm{d}y
\nonumber\\&\le
C_3\left(|v_0|_{6}^{5}
|v_0^{4}u_\varepsilon|_{\frac{6}{5}}
+|v_0|_{6}^{5} |v_0 u_\varepsilon^{4}|_{\frac{6}{5}}\right)
+C_4\left(|u_\varepsilon|_6^5 |v_0^4u_\varepsilon|_{\frac{6}{5}}
+|u_\varepsilon|_6^5
|v_0u_\varepsilon^4|_{\frac{6}{5}}\right)
\nonumber\\& \le C_5 |u_\varepsilon|_{\frac{6}{5}}
+C_6|u_\varepsilon|_{\frac{24}{5}}^{4}
\le C_7\left(\int_{|x|\le
2r_0}|u_\varepsilon|^{\frac{12}{5}}\mathrm{d}x\right)^{\frac{5}{12}}
+C_8 \varepsilon^{\frac{1}{4}}
\nonumber\\&\le C_9 \varepsilon^{\frac{1}{4}}.
\end{align}
So, substituting \eqref{4.116}$-$\eqref{3.146} and \eqref{3.156}
into \eqref{9.5}, by Lemma \ref{lam-64} and $q\in(5,6)$, we
obtain
\begin{equation}\label{6.86}
m_\lambda
\le\mathcal{I}_\lambda(s_\varepsilon v_0-t_\varepsilon
u_\varepsilon)
\le \mathcal{I}_\lambda(v_0)+c^*_\lambda
+C_1\varepsilon^{\frac{1}{4}}
-C_2\varepsilon^{\frac{6-q}{4}}
<c_\lambda+c^*_\lambda,
\end{equation}
as $\varepsilon\to 0$. This completes the proof.
\end{proof}

\subsection{The $(\mathrm{PS})_{m_\lambda}$ condition}

\noindent In what follows, we will show that the functional
$\mathcal{I}_\lambda$ satisfies the $(\mathrm{PS})_{m_\lambda}$
condition.

\begin{lem}\label{lam-116}
Assume that there exists $\lambda^*<0$ such that for all $\lambda\in(\lambda^*,0)$,
$\{u_n\}\subset U_\lambda$ satisfying
\begin{equation*}
\mathcal{I}_\lambda(u_n)\rightarrow m_\lambda\in\left(0,c_\lambda+c^*_\lambda\right),\quad
\mathcal{I}_\lambda'(u_n)\rightarrow0
\end{equation*}
contains a convergent subsequence.
\end{lem}

\begin{proof}
It obtains from $\mathcal{I}_\lambda(u_n)\rightarrow m_\lambda$,
$\mathcal{I}_\lambda'(u_n)\rightarrow 0$ and $q\in(2,6)$ that
\begin{align*}
m_\lambda+1+\|u_n\|
&\ge\mathcal{I}_\lambda(u_n)-\frac{1}{q}\langle
\mathcal{I}_\lambda' (u_n),u_n\rangle
\nonumber\\&=\left(\frac{1}{2}-\frac{1}{q}\right)\|u_n\|^2
+\left(\frac{\lambda}{10}-\frac{\lambda}{q}\right)
\int_{\mathbb{R}^{3}}\phi_{u_n}|u_n|^{5}\mathrm{d}x
+\left(\frac{1}{q}-\frac{1}{6}\right)\int_{\mathbb{R}^{3}}|u_n|^{6}\mathrm{d}x
\\&\ge\left(\frac{1}{2}-\frac{1}{q}\right)\|u_n\|^2,
\end{align*}
which indicates that $\{u_n\}$ is bounded in
$H_r^1(\mathbb{R}^3)$. Then, up to a subsequence if necessary,
still denoted by $\{u_n\}$, we assume that there exists $u\in
H_r^1(\mathbb{R}^3)$ such that for any $r\in[2,6)$,
\begin{equation}\label{9.13}
u_n\rightharpoonup u\ \mathrm{in}\ H_r^1(\mathbb{R}^3),\ \
u_n\rightarrow u\ \mathrm{in}\ L^r(\mathbb{R}^{3}),\ \
u_n(x)\rightarrow u(x)\ \mathrm{a.e.}\ \mathrm{in} \
\mathbb{R}^{3}.
\end{equation}

Firstly, we prove that $\mathcal{I}_\lambda'(u)=0$. It suffices
to verify that $\langle \mathcal{I}_\lambda'(u),\psi\rangle=0$
for all $\psi\in \mathcal{C}_0^\infty(\mathbb{R}^3)$. Observe
that
\begin{align}\label{004}
\langle \mathcal{I}_\lambda'({{u_n}})
-\mathcal{I}_\lambda'({{u}}),\psi\rangle
&=\langle u_n-u,\psi\rangle
+\lambda\int_{\mathbb{R}^{3}}
\left(\phi_{u_n}|u_n|^3u_n-\phi_{u}|u|^3u\right)\psi
\mathrm{d}x
\nonumber\\&\quad
-\int_{\mathbb{R}^{3}}
\left(|u_n|^4u_n-|u|^4u\right)\psi
\mathrm{d}x
-\int_{\mathbb{R}^{3}}
\left(|u_n|^{q-2}u_n-|u|^{q-2}u\right)\psi
\mathrm{d}x.
\end{align}
In view of $u_n\rightharpoonup u$ in $H_r^1(\mathbb{R}^3)$, then
$\langle u_n-u,\psi\rangle\rightarrow 0$. By H\"{o}lder's
inequality and $|a^m-b^m|\le L\max\{a^{m-1},b^{m-1}\}|a-b|$ for
$a,b\ge 0$, $m\ge 1$ and some $L>0$, there hold
\begin{align}\label{9.12}
\left|\int_{\mathbb{R}^{3}}(|u_n|^4u_n-|u|^4u)\psi\mathrm{d}x\right|
&\le \int_{\mathbb{R}^{3}}|u_n|^{4}|u_n-u||\psi|\mathrm{d}x
+\int_{\mathbb{R}^{3}}\Big||u_n|^4-|u|^4\Big||u\psi|\mathrm{d}x
\nonumber\\&\le |\psi|_\infty
|u_n|_5^4\left(\int_{\mathrm{supp}
\psi}|u_n-u|^{5}\mathrm{d}x\right)^{\frac{1}{5}}
\nonumber\\&\quad + C |\psi|_\infty
|u|_5\left(|u_n|_5^3+|u|_5^3\right)\left(\int_{\mathrm{supp}
\psi}|u_n-u|^{5}\mathrm{d}x\right)^{\frac{1}{5}}\rightarrow 0,
\end{align}
as $n\to\infty$. Similarly,
\begin{equation}\label{9.11}
\left|
\int_{\mathbb{R}^{3}}
(|u_n|^{q-2}u_n-|u|^{q-2}u)\psi
\mathrm{d}x
\right|
\rightarrow 0,\ \ \mathrm{as}\ n\to\infty.
\end{equation}
For this, it remains to prove that
\begin{align}\label{67}
\int_{\mathbb{R}^{3}}
\left(\phi_{u_n}|u_n|^3u_n-\phi_{u}|u|^3u\right)\psi
\mathrm{d}x
\rightarrow 0,\ \ \mathrm{as}\ n\to\infty.
\end{align}
In deed, by Proposition \ref{prop-1}$(5)$, we get
$\phi_{u_n}\rightharpoonup \phi_u$ in
$\mathcal{D}^{1,2}(\mathbb{R}^{3})$ and so
 $\phi_{u_n}\rightharpoonup \phi_u$ in $L^{6}(\mathbb{R}^{3})$.
 Then
\begin{align}\label{9.10}
\int_{\mathbb{R}^{3}}
\left(\phi_{u_n}-\phi_{u}\right)|u|^3u\psi
\mathrm{d}x
\rightarrow 0,\ \ \mathrm{as}\ n\to\infty.
\end{align}
Since $u_n\to u$ a.e. in $\mathbb{R}^3$ and
\begin{align*}
\int_{\mathbb{R}^{3}}
\left|\phi_{u_n}\left(|u_n|^3u_n-|u|^3u\right)\right|^{\frac{6}{5}}
\mathrm{d}x
&=\int_{\mathbb{R}^{3}}
|\phi_{u_n}|^{\frac{6}{5}}\big|u+\theta(u_n-u)\big|^{\frac{18}{5}}
\left|u_n-u\right|^{\frac{6}{5}}
\mathrm{d}x
\\&\le C_1 |\phi_{u_n}|_6^{\frac{6}{5}}
\big||u_n|_6+|u|_6\big|^{\frac{18}{5}}
\left|u_n-u\right|_6^{\frac{6}{5}}
\\&\le C_2,
\end{align*}
where $0<\theta<1$, we see that
$\phi_{u_n}\left(|u_n|^3u_n-|u|^3u\right)\rightharpoonup 0$ in
$L^{\frac{6}{5}}(\mathbb{R}^3)$ and thus
\begin{align*}
\int_{\mathbb{R}^{3}}
\phi_{u_n}\left(|u_n|^3u_n-|u|^3u\right)\psi
\mathrm{d}x
\rightarrow 0,\ \ \mathrm{as}\ n\to\infty,
\end{align*}
this together with \eqref{9.10}, we conclude that \eqref{67}
holds. Substituting \eqref{9.12}$-$\eqref{67} into \eqref{004},
using the fact that $u_n\rightharpoonup u$ in
$H_r^1(\mathbb{R}^3)$, it holds that
$$\langle \mathcal{I}_\lambda'(u),\psi\rangle
=\lim\limits_{n\to \infty}\langle
\mathcal{I}_\lambda'(u_n),\psi\rangle=0,$$
for any $\psi\in \mathcal{C}_0^\infty(\mathbb{R}^3)$, which means
that
$\mathcal{I}_\lambda'(u)=0$.

Secondly, we will show that $u\neq 0$. Suppose by contradiction
that $u\equiv 0$, that is, $u^+\equiv 0$ and $u^-\equiv 0$. By applying \eqref{2.3}, \eqref{9.13}, $\mathcal{I}_\lambda'(u_n)\rightarrow0$ and the
Hardy-Littlewood-Sobolev inequality (see Proposition
\ref{prohlsi}), we infer that
\begin{align*}
\lim\limits_{n\to \infty}\|u_n^+\|^2
&=-\lambda\lim\limits_{n\to \infty}
\int_{\mathbb{R}^{3}}\phi_{u_n} |u_n^+|^5\mathrm{d}x
+\lim\limits_{n\to
\infty}\int_{\mathbb{R}^{3}}|u_n^+|^{6}\mathrm{d}x
+\lim\limits_{n\to
\infty}\int_{\mathbb{R}^{3}}|u_n^+|^{q}\mathrm{d}x
\nonumber\\&=-\lambda\lim\limits_{n\to \infty}
\int_{\mathbb{R}^{3}}\phi_{u_n^+} |u_n^+|^5\mathrm{d}x
-\lambda\lim\limits_{n\to \infty}
\int_{\mathbb{R}^{3}}\phi_{u_n^-} |u_n^+|^5\mathrm{d}x
+\lim\limits_{n\to
\infty}\int_{\mathbb{R}^{3}}|u_n^+|^{6}\mathrm{d}x
\nonumber\\&\le -\lambda C_1 \lim\limits_{n\to\infty}
|u_n^+|_6^{10}
-\lambda C_1 \lim\limits_{n\to\infty}
|u_n^-|_6^5|u_n^+|_6^5
+\lim\limits_{n\to \infty}|u_n^+|^{6}_6
\nonumber\\&\le -\lambda C_1S^{-5}
\lim\limits_{n\to\infty}\|u_n^+\|^{10}
-\lambda C_1 S^{-5}\lim\limits_{n\to\infty}\|u_n^-\|^5\|u_n^+\|^5
+S^{-3}\lim\limits_{n\to \infty}\|u_n^+\|^{6}.
\end{align*}
It gets from Lemma \ref{lam-126}$(3)$ that
$\Lambda_2\le\lim\limits_{n\to \infty}\|u_{n}^\pm\|\le \Lambda_1$. Thus, the above
inequality turns into
\begin{equation*}
0<\lim\limits_{n\to \infty}\|u_n^+\|^2
\le -\lambda C_2 S^{-5} \lim\limits_{n\to\infty}\|u_n^+\|^{10}
+S^{-3}\lim\limits_{n\to \infty}\|u_n^+\|^{6}.
\end{equation*}
Let $0<t:=\lim\limits_{n\to \infty}\|u_n^+\|^2$, then
$0<t\le -\lambda C_2 S^{-5}t^5+S^{-3}t^3$. A simple calculation
yields that
$$t^2\ge t_*^2:=\frac{-S^{-3}+\sqrt{S^{-6}-4\lambda
C_2S^{-5}}}{-2\lambda C_2 S^{-5}},$$
that is,
\begin{equation}\label{9.14}
\lim\limits_{n\to \infty}\|u_n^+\|^2
\ge t_*
=\left(\frac{-S^{-3}+\sqrt{S^{-6}-4\lambda C_2S^{-5}}}{-2\lambda
C_2 S^{-5}}\right)^{\frac{1}{2}}.
\end{equation}
Similarly,
\begin{equation}\label{9.15}
\lim\limits_{n\to \infty}\|u_n^-\|^2
\ge t_*
=\left(\frac{-S^{-3}+\sqrt{S^{-6}-4\lambda C_2S^{-5}}}{-2\lambda
C_2 S^{-5}}\right)^{\frac{1}{2}}.
\end{equation}
Then, we derive from $\lambda<0$, \eqref{9.13}, \eqref{9.14} and
\eqref{9.15} that
\begin{align}\label{9.16}
m_\lambda&=\lim\limits_{n\to \infty}\mathcal{I}_\lambda(u_n)
\nonumber\\&=\frac{1}{2}\lim\limits_{n\to \infty}\|u_n\|^2
+\frac{\lambda}{10}\lim\limits_{n\to \infty}
\int_{\mathbb{R}^{3}}\phi_{u_n} |u_n|^5\mathrm{d}x
-\frac{1}{6}\lim\limits_{n\to \infty}
\int_{\mathbb{R}^{3}}|u_n|^{6}\mathrm{d}x
-\frac{1}{q}\lim\limits_{n\to \infty}
\int_{\mathbb{R}^{3}}|u_n|^{q}\mathrm{d}x
\nonumber\\&=\frac{1}{2}\lim\limits_{n\to \infty}\|u_n\|^2
+\frac{\lambda}{10}\lim\limits_{n\to \infty}
\int_{\mathbb{R}^{3}}\phi_{u_n} |u_n|^5\mathrm{d}x
-\frac{1}{6}\lim\limits_{n\to \infty}
\int_{\mathbb{R}^{3}}|u_n|^{6}\mathrm{d}x
\nonumber\\&= \frac{1}{2}\lim\limits_{n\to \infty}\|u_n\|^2
+\left(\frac{\lambda}{10}-\frac{\lambda}{6}\right)
\lim\limits_{n\to \infty}
\int_{\mathbb{R}^{3}}\phi_{u_n} |u_n|^5\mathrm{d}x
\nonumber\\&\quad+\frac{1}{6}
\left(\lim\limits_{n\to \infty}
\lambda\int_{\mathbb{R}^{3}}\phi_{u_n} |u_n|^5\mathrm{d}x
-\lim\limits_{n\to \infty}
\int_{\mathbb{R}^{3}}|u_n|^{6}\mathrm{d}x\right)
\nonumber\\&= \frac{1}{2}\lim\limits_{n\to \infty}\|u_n\|^2
+\left(\frac{\lambda}{10}-\frac{\lambda}{6}\right)
\lim\limits_{n\to \infty}
\int_{\mathbb{R}^{3}}\phi_{u_n} |u_n|^5\mathrm{d}x
-\frac{1}{6}\lim\limits_{n\to \infty}\|u_n\|^2
\nonumber\\&\ge \frac{1}{2}\lim\limits_{n\to \infty}\|u_n\|^2
-\frac{1}{6}\lim\limits_{n\to \infty}\|u_n\|^2
\nonumber\\&= \frac{1}{3}\lim\limits_{n\to \infty}\|u_n\|^2
= \frac{1}{3}\lim\limits_{n\to \infty}\|u_n^+\|^2
+\frac{1}{3}\lim\limits_{n\to \infty}\|u_n^-\|^2
\nonumber\\&\ge \frac{2}{3}t_*
\ge
\frac{2}{3}\left(\frac{-S^{-3}+\sqrt{S^{-6}-4\lambda
C_2S^{-5}}}{-2\lambda C_2 S^{-5}}\right)^{\frac{1}{2}}.
\end{align}
As $\lambda\to 0^-$, we see  that
\begin{align*}
\lim\limits_{\lambda\to 0^-}
\frac{-S^{-3}+\sqrt{S^{-6}-4\lambda C_2S^{-5}}}{-2\lambda C_2
S^{-5}}
&=\lim\limits_{\lambda\to 0^-}
\frac{-2 C_2S^{-5}\left(S^{-6}-4\lambda
C_2S^{-5}\right)^{-\frac{1}{2}}}{-2 C_2 S^{-5}}
\\&=\lim\limits_{\lambda\to 0^-} \left(S^{-6}-4\lambda
C_2S^{-5}\right)^{-\frac{1}{2}}
\\&= S^{3}.
\end{align*}
Therefore, as $\lambda\to 0^-$, we infer from \eqref{9.16} and
the above inequality that
$$c_0+c_0^*>m_0\ge \frac{2}{3}S^{\frac{3}{2}},$$
where $c_0^*=\frac{1}{3}S^{\frac{3}{2}}$ and
$c_0<\frac{1}{3}S^{\frac{3}{2}}$ (obtained in Remark
\ref{rem-16}), which leads to a contradiction. Hence, there exists $\lambda^*<0$ such that for all $\lambda\in(\lambda^*,0)$, we infer that  $u^+\not\equiv 0$ and $u^-\not\equiv 0$, that is, $u\neq 0$. Which together with $\mathcal{I}_\lambda'(u)=0$, it holds that $u\in \mathcal{N}_\lambda$ and $\mathcal{I}_\lambda(u)\ge c_\lambda$.

Lastly, we will prove that $u_n\to u$ in $H^1_r(\mathbb{R}^3)$.
Denote $v_n:=u_n-u$, it yields from Br\'{e}zis-Lieb lemma (see
\cite[Theorem 1]{BH83}), Proposition \ref{prop-1}($5$) and the
compactness embedding of $H_r^1(\mathbb{R}^3)\hookrightarrow
L^r(\mathbb{R}^3)$ for $r\in[2,6)$ that
\begin{align}\label{4.306}
m_\lambda&=\mathcal{I}_\lambda(u_n)+o(1)
\nonumber\\&=\frac{1}{2}\|u_n\|^2
+\frac{\lambda}{10}\int_{\mathbb{R}^{3}}\phi_{u_n}
|u_n|^5\mathrm{d}x
-\frac{1}{6}\int_{\mathbb{R}^{3}}|u_n|^{6}\mathrm{d}x
-\frac{1}{q}\int_{\mathbb{R}^{3}}|u_n|^{q}\mathrm{d}x+o(1)\nonumber\\&
=\frac{1}{2}\|u\|^2+\frac{1}{2}\|v_n\|^2
+\frac{\lambda}{10}\int_{\mathbb{R}^{3}}\phi_{u} |u|^5\mathrm{d}x
+\frac{\lambda}{10}\int_{\mathbb{R}^{3}}\phi_{v_n}
|v_n|^5\mathrm{d}x
\nonumber\\&\quad-\frac{1}{6}\int_{\mathbb{R}^{3}}|u|^{6}\mathrm{d}x
-\frac{1}{6}\int_{\mathbb{R}^{3}}|v_n|^{6}\mathrm{d}x
-\frac{1}{q}\int_{\mathbb{R}^{3}}|u|^{q}\mathrm{d}x+o(1)\nonumber\\
&=\mathcal{I}_\lambda(u)+\frac{1}{2}\|v_n\|^2
+\frac{\lambda}{10}\int_{\mathbb{R}^{3}}\phi_{v_n}
|v_n|^5\mathrm{d}x
-\frac{1}{6}\int_{\mathbb{R}^{3}}|v_n|^{6}\mathrm{d}x+o(1),
\end{align}
and
\begin{align}\label{4.316}
0&=\langle \mathcal{I}_\lambda'(u_n),u_n\rangle+o(1)\nonumber\\&
=\|u_n\|^2
+\lambda\int_{\mathbb{R}^{3}}\phi_{u_n} |u_n|^5\mathrm{d}x
-\int_{\mathbb{R}^{3}}|u_n|^{6}\mathrm{d}x
-\int_{\mathbb{R}^{3}}|u_n|^{q}\mathrm{d}x+o(1)\nonumber\\&
=\|u\|^2+\|v_n\|^2
+\lambda\int_{\mathbb{R}^{3}}\phi_{u} |u|^5\mathrm{d}x
+\lambda\int_{\mathbb{R}^{3}}\phi_{v_n} |v_n|^5\mathrm{d}x
\nonumber\\&\quad-\int_{\mathbb{R}^{3}}|u|^{6}\mathrm{d}x
-\int_{\mathbb{R}^{3}}|v_n|^{6}\mathrm{d}x\nonumber
-\int_{\mathbb{R}^{3}}|u|^{q}\mathrm{d}x+o(1)
\nonumber\\&=\langle \mathcal{I}_\lambda'(u),u\rangle
+\|v_n\|^2
+\lambda\int_{\mathbb{R}^{3}}\phi_{v_n} |v_n|^5\mathrm{d}x
-\int_{\mathbb{R}^{3}}|v_n|^{6}\mathrm{d}x+o(1)
\nonumber\\&=\|v_n\|^2
+\lambda\int_{\mathbb{R}^{3}}\phi_{v_n} |v_n|^5\mathrm{d}x
-\int_{\mathbb{R}^{3}}|v_n|^{6}\mathrm{d}x+o(1).
\end{align}
If $v_n\rightarrow0$ in $H_r^1(\mathbb{R}^3)$, the proof of Lemma
\ref{lam-116} is completed. So we suppose by contradiction that
$v_n\rightharpoonup0$ and $v_n\nrightarrow0$ in
$H_r^1(\mathbb{R}^3)$. Then by \eqref{4.316}, we may assume that,
for $n$ large enough,
\begin{align}\label{5.76}
\|v_n\|^2\rightarrow l,\quad
\lambda\int_{\mathbb{R}^{3}}\phi_{v_n} |v_n|^5\mathrm{d}x\to
a,\quad
\int_{\mathbb{R}^{3}}|v_n|^{6}\mathrm{d}x\rightarrow b.
\end{align}
Notice that for any $\tau>0$, by $-\Delta \phi_{v_n}=|v_n|^5$, it
holds that
\begin{align*}
\int_{\mathbb{R}^3}|v_n|^6\mathrm{d}x
=\int_{\mathbb{R}^3}-\Delta \phi_{v_n}|v_n|\mathrm{d}x
&\le \frac{1}{2\tau^{2}}\int_{\mathbb{R}^3}|\nabla
\phi_{v_n}|^2\mathrm{d}x
+\frac{\tau^{2}}{2}\int_{\mathbb{R}^3}|\nabla |v_n||^2\mathrm{d}x
\\&\le
\frac{1}{2\tau^{2}}\int_{\mathbb{R}^3}\phi_{v_n}|v_n|^5\mathrm{d}x
+\frac{\tau^{2}}{2}\int_{\mathbb{R}^3}|\nabla v_n|^2\mathrm{d}x.
\end{align*}
Thus, passing to the limit as $n\to \infty$, it holds that
\begin{equation}\label{6.36}
a+l=b\le \frac{a}{2\lambda\tau^{2}}
+\frac{\tau^{2}}{2}l.
\end{equation}
With the help of $\lambda<0$, choosing
\[
\tau^2=\frac{1-\sqrt{1-4\lambda}}{2\lambda}>0,
 \]
which together with \eqref{6.36} implies that
$a\leq \frac{-2 \lambda+1-\sqrt{1-4\lambda }}{2\lambda }l$. It
follows from \eqref{4.306}$-$\eqref{5.76} that
\begin{align}\label{5.86}
\frac{1}{2}\|v_n\|^2
+\frac{\lambda}{10}\int_{\mathbb{R}^{3}}\phi_{v_n}
|v_n|^5\mathrm{d}x
-\frac{1}{6}\int_{\mathbb{R}^{3}}|v_n|^{6}\mathrm{d}x+o(1)
=\frac{l}{3}-\frac{a}{15}\ge \frac{12 \lambda-1+\sqrt{1-4\lambda
}}{30\lambda }l.
\end{align}
On the other hand, it follows from \eqref{2.3} that
\begin{equation}\label{7.6}
\int_{\mathbb{R}^{3}}|v_n|^{6}\mathrm{d}x
\le S^{-3}\|v_n\|_{\mathcal{D}^{1,2}(\mathbb{R}^3)}^{6}
\le S^{-3}\|v_n\|^{6}.
\end{equation}
So, it gives that
\begin{equation}\label{6.16}
b\le S^{-3}l^{3}.
\end{equation}
By \eqref{2.3} and H\"{o}lder inequality, we obtain that
\begin{align*}
\int_{\mathbb{R}^{3}}\phi_{v_n} |v_n|^5\mathrm{d}x
&\le
\left(\int_{\mathbb{R}^{3}}|\phi_{v_n}|^6\mathrm{d}x\right)^{\frac{1}{6}}
\left(\int_{\mathbb{R}^{3}}|v_n|^6\mathrm{d}x\right)^{\frac{5}{6}}
\\&\le S^{-\frac{1}{2}}\left(\int_{\mathbb{R}^{3}}|\nabla
\phi_{v_n}|^2\mathrm{d}x\right)^{\frac{1}{2}}
\left(\int_{\mathbb{R}^{3}}|v_n|^6\mathrm{d}x\right)^{\frac{5}{6}}
\\&\le S^{-\frac{1}{2}}\left(\int_{\mathbb{R}^{3}}\phi_{v_n}
|v_n|^5\mathrm{d}x\right)^{\frac{1}{2}}
\left(\int_{\mathbb{R}^{3}}|v_n|^6\mathrm{d}x\right)^{\frac{5}{6}},
\end{align*}
which directly yields from \eqref{7.6} that
\begin{equation*}
\int_{\mathbb{R}^{3}}\phi_{v_n} |v_n|^5\mathrm{d}x
\le S^{-1}
\left(\int_{\mathbb{R}^{3}}|v_n|^6\mathrm{d}x\right)^{\frac{5}{3}}
\le S^{-6}\|v_n\|^{10}.
\end{equation*}
Then, it gives that
\begin{equation}\label{6.26}
a
\ge \lambda S^{-6}l^5.
\end{equation}
Thus, it follows from \eqref{4.316}, \eqref{5.76}, \eqref{6.16}
and \eqref{6.26} that
\[
l
=b-a
\le S^{-3}l^{3}-\lambda S^{-6}l^5.
 \]
 Thus, $l=0$ or $l^2\ge \frac{1-\sqrt{1-4\lambda}}{2\lambda}S^3$.
 If $l^2\ge \frac{1-\sqrt{1-4\lambda}}{2\lambda}S^3$, combining
 with \eqref{5.86}, it yields that
\begin{align*}
\frac{1}{2}\|v_n\|^2
+\frac{\lambda}{10}\int_{\mathbb{R}^{3}}\phi_{v_n}
|v_n|^5\mathrm{d}x
-\frac{1}{6}\int_{\mathbb{R}^{3}}|v_n|^{6}\mathrm{d}x
+o(1)
&  \ge \frac{12 \lambda-1+\sqrt{1-4\lambda }}{30\lambda }l
\\&
\ge \frac{12 \lambda-1+\sqrt{1-4\lambda }}{30\lambda }
\left(\frac{1-\sqrt{1-4\lambda}}{2\lambda}\right)^{\frac{1}{2}}S^{\frac{3}{2}}
\\ &=c^*_\lambda.
\end{align*}
Hence,
\begin{equation*}
m_\lambda=\mathcal{I}_\lambda(u_n)+o(1)
=\mathcal{I }_\lambda(u)+\left(\frac{1}{2}\|v_n\|^2
+\frac{\lambda}{10}\int_{\mathbb{R}^{3}}\phi_{v_n}
|v_n|^5\mathrm{d}x
-\frac{1}{6}\int_{\mathbb{R}^{3}}|v_n|^{6}\mathrm{d}x\right)+o(1)
\ge c_\lambda+c^*_\lambda,
\end{equation*}
which leads to a contradiction with our assumption
$m_\lambda\in(0,c_\lambda+c^*_\lambda)$.
\end{proof}

Now, we complete the proof of Theorem \ref{thm-2}.

\begin{proof}[\textbf{Proof of Theorem \ref{thm-2}}] From Lemma
\ref{lam-76}, we know that there exists a sequence
$\{u_n\}\subset U_\lambda$ satisfying
$\mathcal{I}_\lambda(u_n)\rightarrow m_\lambda$ and
$\mathcal{I}_\lambda'(u_n)\rightarrow0$ as $n\rightarrow\infty$.
Combining Lemma \ref{lam-86} with Lemma \ref{lam-116}, we obtain
that $\{u_n\}$ contains a convergent subsequence, still denoted
by $\{u_n\}$. Then there exists $u_\lambda\in
H_r^1(\mathbb{R}^3)$ such that $u_n\rightarrow u_\lambda$ in
$H_r^1(\mathbb{R}^3)$ as $n\rightarrow\infty$, and by the
continuity of  $\mathcal{I}_\lambda$ and $\mathcal{I}_\lambda'$,
we see that $\mathcal{I}_\lambda(u_\lambda)=m_\lambda$ and
$\mathcal{I}_\lambda'(u_\lambda)=0$. Furthermore, from
$\{u_n\}\subset U_\lambda$, we have
$\frac{1}{2}<l_\lambda(u_n^+,u_n^-)<\frac{3}{2}$ and
$\frac{1}{2}<l_\lambda(u_n^-,u_n^+)<\frac{3}{2}$ by \eqref{4.26},
then
\begin{align*}
\frac{1}{2}\|u_n^+\|^2
&<\int_{\mathbb{R}^{3}}|u_n^+|^6\mathrm{d}x
+\int_{\mathbb{R}^{3}}|u_n^+|^q\mathrm{d}x
-\lambda\int_{\mathbb{R}^{3}}\phi_{u_n^+} |u_n^+|^5\mathrm{d}x
-\lambda\int_{\mathbb{R}^{3}}\phi_{u_n^-} |u_n^+|^5\mathrm{d}x
\\&\le
C_1\|u_n^+\|^{6}
+C_2\|u_n^+\|^{q}
+C_3\|u_n^+\|^{10}
+C_4\|u_n^-\|^{5}\|u_n^+\|^{5}
.
\end{align*}
This together with $u^-_n\to u^-_\lambda$ in
$H_r^1(\mathbb{R}^3)$, we conclude that
\begin{align*}
\frac{1}{2}\|u_n^+\|^2
<C_1\|u_n^+\|^{6}
+C_2\|u_n^+\|^{q}
+C_3\|u_n^+\|^{10}
+C_5\|u_n^+\|^{5},
\end{align*}
then by $q\in(2,6)$, there exists $\varrho>0$ such that
$\|u_n^+\|\ge\varrho$, which implies that
$$\|u_\lambda^+\|=\lim\limits_{n\to \infty}\|u_n^+\|\ge
\varrho>0.$$
Similarly, $\|u_\lambda^-\|\ge \varrho>0$. Thus, $u_\lambda$ is a
least energy radial sign-changing solution of \eqref{1.6}.
\end{proof}


\section{Least energy radial sign-changing solutions for the case
$\lambda=0$}\label{4}
\noindent In this section, we are interested in the existence of
least energy radial sign-changing solutions for problem \eqref{1.7}, and
prove Corollary \ref{thm-5}. Before proving Corollary \ref{thm-5}, we
give some definitions firstly. Define the energy functional
$\mathcal{I}_0$ associating with problem \eqref{1.7} by
\begin{equation*}
\mathcal{I}_0(u):=\frac{1}{2}\|u\|^2
-\frac{1}{6}\int_{\mathbb{R}^{3}}|u|^{6}\mathrm{d}x
-\frac{1}{q}\int_{\mathbb{R}^{3}}|u|^{q}\mathrm{d}x.
\end{equation*}
Meanwhile, let us define
\begin{equation*}
c_0:=\inf_{u\in\mathcal{N}_0}\mathcal{I}_0(u),\quad
m_0:=\inf_{u\in\mathcal{M}_0}\mathcal{I}_0(u),
\end{equation*}
where
$$\mathcal{N}_0:=\left\{u\in
H_r^1(\mathbb{R}^3)\setminus\{0\}:\langle
\mathcal{I}_0'(u),u\rangle=0\right\},$$
and
$$\mathcal{M}_0:=\left\{u\in
H_r^1(\mathbb{R}^3):u^{\pm}\neq0,\langle
\mathcal{I}_0'(u),u^\pm\rangle
=0\right\}.$$

We first show the following lemma, in view of our
nonlinearity, which can be directly conclude from \cite[Lemma
2.2]{CXP21} with $\lambda\equiv 0$.
So we omit it here.

\begin{lem}\label{lam-66}
Let $v\in H_r^{1}(\mathbb{R}^3)$ with $v^\pm\neq0$, then there
exists a unique pair $(s_v,t_v)\in (0,\infty)\times(0,\infty)$
such that $s_vv^++t_vv^-\in \mathcal{M}_0$. Moreover,
$$\mathcal{I }_0(s_vv^++t_vv^-)
=\max\limits_{s,t\ge0}\,\mathcal{I}_0(sv^++tv^-).$$
\end{lem}

\begin{lem}\label{lam-67}
If $\langle \mathcal{I}_0'(v),v^\pm\rangle=0$, then $s_v=t_v=1$,
where $(s_v,t_v)$ is obtained by Lemma \ref{lam-66}.

\end{lem}

\begin{proof} Since $s_vv^++t_vv^-\in \mathcal{M}_0$, then
\begin{equation}\label{2.23}
{s_v^2}\|v^+\|^2
= s_v^{6}\int_{\mathbb{R}^{3}}|v^+|^{6}\mathrm{d}x
+s_v^{q}\int_{\mathbb{R}^{3}}|v^+|^{q}\mathrm{d}x.
\end{equation}
Since $\langle \mathcal{I}_0'(v),v^+\rangle=0$, it yields that
\begin{equation}\label{2.24}
\|v^+\|^2
=\int_{\mathbb{R}^{3}}|v^+|^{6}\mathrm{d}x
+\int_{\mathbb{R}^{3}}|v^+|^{q}\mathrm{d}x.
\end{equation}
Combining \eqref{2.23} with \eqref{2.24}, we obtain that
\begin{equation*}
\left(1-{s_v^{2-q}}\right)\|v^+\|^2
=\left(1-s_v^{6-q}\right)\int_{\mathbb{R}^{3}}|v^+|^{6}\mathrm{d}x.
\end{equation*}
This together with $q\in(2,6)$ derives that $s_v=1$. Similarly,
$t_v=1$. So, the proof is complete.
\end{proof}

\begin{proof}[\textbf{Proof of Corollary \ref{thm-5}}]
Similar arguments as \cite[Theorem 1.1]{ZZH21} with
$\phi\equiv0$,
we obtain that problem \eqref{1.7} has a least energy radial
sign-changing solution $z_0$.

In what follows, we will prove that the least energy radial
sign-changing solution $z_0$ has exactly two nodal domains.
Suppose by contradiction that $z_0$ has at least three nodal
domains satisfying
\begin{equation*}
z_0=z_1+z_2+z_3
\end{equation*}
with $z_i\neq0,\ z_1\ge0,\ z_2\le0$ and supp$(z_i)\ \cap\
$supp$(z_j)=\emptyset$, for $i\neq j,\ i,j=1,2,3$
and
\begin{align*}
\langle \mathcal{I}_0'(z_0),z_i\rangle=0,\ \ \mathrm{for}\
i=1,2,3.
\end{align*}
Setting $w:=z_1+z_2$, we obtain that $w^+=z_1,\ w^-=z_2$, i.e.,
$w^\pm\neq0$. By applying Lemma \ref{lam-66}, there exists a
unique pair $(s_w,t_w)\in (0,\infty)\times (0,\infty)$ such that
$$s_ww^++t_ww^-=s_wz_1+t_wz_2\in\mathcal{M}_0\quad
\mathrm{and}\quad
\mathcal{I}_0(s_wz_1+t_wz_2)\ge m_0.$$
Using the fact that $\langle \mathcal{I }_0'(z_0),z_i\rangle=0$
for $i=1,2,3$, it follows that $\langle \mathcal{I}
_0'(w),w^\pm\rangle=0$, then $s_w=t_w=1$ by Lemma \ref{lam-67}.
On the other hand, we deduce that
\begin{align*}
0=\frac{1}{2}\langle \mathcal{I}_0'(z_0),z_3\rangle
=\frac{1}{2}\|z_3\|^2
-\frac{1}{2}\int_{\mathbb{R}^{3}}|z_3|^{6}\mathrm{d}x
-\frac{1}{2}\int_{\mathbb{R}^{3}}|z_3|^{q}\mathrm{d}x
<\mathcal{I}_0(z_3).
\end{align*}
Then, we see that
\begin{align*}
m_{0}
&\le \mathcal{I}_0(s_wz_1+t_wz_2)
\\&=\mathcal{I}_0(s_wz_1+t_wz_2)
-\frac{1}{q}
\langle \mathcal{I}_0'(s_w z_1+t_wz_2),s_wz_1+t_wz_2\rangle
\\&=\left(\frac{1}{2}-\frac{1}{q}\right)\|s_w z_1\|^2
+\left(\frac{1}{q}-\frac{1}{6}\right)\int_{\mathbb{R}^{3}}|s_w
z_1|^{6}\mathrm{d}x
\\&\quad+\left(\frac{1}{2}-\frac{1}{q}\right)\|t_w z_2\|^2
+\left(\frac{1}{q}-\frac{1}{6}\right)\int_{\mathbb{R}^{3}}|t_w
z_2|^{6}\mathrm{d}x
\\&=\left(\frac{1}{2}-\frac{1}{q}\right)\|z_1\|^2
+\left(\frac{1}{q}-\frac{1}{6}\right)
\int_{\mathbb{R}^{3}}|z_1|^{6}\mathrm{d}x
\\&\quad+\left(\frac{1}{2}-\frac{1}{q}\right)\|z_2\|^2
+\left(\frac{1}{q}-\frac{1}{6}\right)
\int_{\mathbb{R}^{3}}|z_2|^{6}\mathrm{d}x
\\&=\mathcal{I}_0(z_1)
-\frac{1}{q}\langle \mathcal{I}_0'(z_1),z_1\rangle
+\mathcal{I}_0(z_2)
-\frac{1}{q}\langle \mathcal{I}_0'(z_2),z_2\rangle
\\&=\mathcal{I}_0(z_1)
+\mathcal{I}_0(z_2)
\\&<\mathcal{I}_0(z_1)
+\mathcal{I}_0(z_2)
+\mathcal{I}_0(z_3)
\\&=\mathcal{I}_0(z_0)=m_{0} ,
\end{align*}
which is a contradiction. That is, $z_3=0$, and $z_0$ has exactly
two nodal domains.

Lastly, it remains to show that $m_0\ge 2c_0$. Similar arguments
as Lemma \ref{lam-66}, there exist $\bar{s},\bar{t}>0$ such that
$\bar{s}z_0^+,\bar{t}z_0^-\in \mathcal{N}_0$. Then, it follows
from Lemma \ref{lam-66} that
\begin{align*}
m_0
=\mathcal{I}_0(z_0)
\ge \mathcal{I}_0(\bar{s}z_{0}^++\bar{t}z_{0}^-)
=\mathcal{I}_0(\bar{s}z_{0}^+)
+\mathcal{I}_0(\bar{t}z_{0}^-)
\ge 2c_{0}.
\end{align*}
Thus, the proof of Corollary \ref{thm-5} is completed.
\end{proof}

\begin{rem}\label{rem-16}
Similar arguments as \cite[Lemma 3.3]{ZZH21} with $\phi\equiv0$,
which gives that $m_0<c_0+\frac{1}{3}S^{\frac{3}{2}}$. This
together with Corollary \ref{thm-5}, it holds that $m_0\ge 2c_0$.
Thus, $c_0< \frac{1}{3}S^{\frac{3}{2}}$ and $m_0<
\frac{2}{3}S^{\frac{3}{2}}$.
\end{rem}

\section{Asymptotic behavior of sign-changing solutions}\label{6}

\noindent Now, we consider the asymptotic behavior of $u_\lambda$
as $\lambda\to 0^-$, and prove Theorem \ref{thm-4}. In what
follows, we regard $\lambda<0$ as a parameter in system
\eqref{1.6} and show the relationship between the case
$\lambda<0$ and $\lambda=0$ in system \eqref{1.6}. We first present the following lemma, which will be used in the proof of Theorem \ref{thm-4}.


\begin{lem}\label{lam-446}
Let $m_\lambda$ be defined by \eqref{9.26}, then $0<m_\lambda<\frac{2}{3}S^{\frac{3}{2}}.$
\end{lem}

\begin{proof}
We first claim that $m_\lambda<c_\lambda+\frac{1}{3}S^{\frac{3}{2}}$. We infer from \eqref{4.116} in Lemma \ref{lam-86} that
\begin{align*}
\Pi_1
&\le\max\limits_{t\ge 0}
\left(\frac{t^2}{2}\|
u_\varepsilon\|_{\mathcal{D}^{1,2}(\mathbb{R}^3)}^2
+\frac{\lambda t^{10}}{10}
\int_{\mathbb{R}^{3}}\phi_{ u_\varepsilon}|
u_\varepsilon|^{5}\mathrm{d}x
-\frac{t^6}{6}
\int_{\mathbb{R}^{3}}| u_\varepsilon|^{6}\mathrm{d}x\right)
\\&\le \max\limits_{t\ge 0}
\left(\frac{t^2}{2}\|
u_\varepsilon\|_{\mathcal{D}^{1,2}(\mathbb{R}^3)}^2
-\frac{t^6}{6}
\int_{\mathbb{R}^{3}}| u_\varepsilon|^{6}\mathrm{d}x\right)
\\&=\frac{1}{3}
\frac{\|u_\varepsilon\|_{\mathcal{D}^{1,2}(\mathbb{R}^3)}^3}
{\left(\int_{\mathbb{R}^{3}}| u_\varepsilon|^{6}\mathrm{d}x\right)^{\frac{1}{2}}}
=\frac{1}{3}S^{\frac{3}{2}}+O(\varepsilon^{\frac{1}{2}}),
\end{align*}
provided that $\lambda<0$. The other part of the proof is the same as Lemma \ref{lam-86}. In view of this, the claim holds. In order to complete the proof of Lemma \ref{lam-446}, it remains to prove $c_\lambda<\frac{1}{3}S^{\frac{3}{2}}$, where $c_\lambda$ defined by \eqref{9.25}.

Indeed, it is obvious that $\mathcal{I}_\lambda(\alpha u_\varepsilon)>0$ for $\alpha>0$ small, and $\mathcal{I}_\lambda(\alpha u_\varepsilon)\to -\infty$ for $\alpha \to \infty$, where $u_\varepsilon$ defined in Subsection \ref{sub-5.3}. Hence, there exists $\alpha_\varepsilon>0$ such that $\mathcal{I}_\lambda(\alpha_\varepsilon u_\varepsilon)
=\max_{\alpha\ge 0}\mathcal{I}_\lambda(\alpha u_\varepsilon)$. It follows from Lemma \ref{lam-64} that $c_\lambda\le \mathcal{I}_\lambda(\alpha_\varepsilon u_\varepsilon)$. Moreover, it is easy to verify that $\alpha_\varepsilon$ contains in a bounded interval. In view of this, it suffices to show that $\mathcal{I}_\lambda(\alpha_\varepsilon u_\varepsilon)<\frac{1}{3}S^{\frac{3}{2}}$. We infer from \eqref{3.86}, $q\in(5,6)$ and $\lambda<0$ that
\begin{align*}
\mathcal{I}_\lambda(\alpha_\varepsilon u_\varepsilon)
&=\frac{1}{2}\|\alpha_\varepsilon u_\varepsilon\|_{\mathcal{D}^{1,2}(\mathbb{R}^3)}^2
+\frac{\lambda}{10}\int_{\mathbb{R}^{3}}\phi_{\alpha_\varepsilon u_\varepsilon}|\alpha_\varepsilon u_\varepsilon|^{5}\mathrm{d}x
-\frac{1}{6}\int_{\mathbb{R}^{3}}|\alpha_\varepsilon u_\varepsilon|^{6}\mathrm{d}x
\\&\quad +\frac{1}{2}
\int_{\mathbb{R}^{3}}|\alpha_\varepsilon u_\varepsilon|^{2}
\mathrm{d}x
-\frac{1}{q}
\int_{\mathbb{R}^{3}}|\alpha_\varepsilon u_\varepsilon|^{q}
\mathrm{d}x
\\&\le\max\limits_{\alpha\ge 0}
\left(\frac{\alpha^2}{2}\|
u_\varepsilon\|_{\mathcal{D}^{1,2}(\mathbb{R}^3)}^2
+\frac{\lambda \alpha^{10}}{10}
\int_{\mathbb{R}^{3}}\phi_{ u_\varepsilon}|
u_\varepsilon|^{5}\mathrm{d}x
-\frac{\alpha^6}{6}
\int_{\mathbb{R}^{3}}| u_\varepsilon|^{6}\mathrm{d}x\right)
+C_1 \varepsilon^{\frac{1}{2}}
-C_2 \varepsilon^{\frac{6-q}{4}}
\\&\le \max\limits_{\alpha\ge 0}
\left(\frac{\alpha^2}{2}\|
u_\varepsilon\|_{\mathcal{D}^{1,2}(\mathbb{R}^3)}^2
-\frac{\alpha^6}{6}
\int_{\mathbb{R}^{3}}| u_\varepsilon|^{6}\mathrm{d}x\right)
+C_1 \varepsilon^{\frac{1}{2}}
-C_2 \varepsilon^{\frac{6-q}{4}}
\\&\le \frac{1}{3}S^{\frac{3}{2}}
+C_1\varepsilon^{\frac{1}{2}}
-C_2\varepsilon^{\frac{6-q}{4}}
<\frac{1}{3}S^{\frac{3}{2}},
\end{align*}
as $\varepsilon\to 0$. Therefore, $m_\lambda<\frac{2}{3}S^{\frac{3}{2}}$, and we complete the proof of Lemma \ref{lam-446}.
\end{proof}

\begin{proof}[\textbf{Proof of Theorem \ref{thm-4}}] Recall that
$u_\lambda\in H_r^1(\mathbb{R}^3)$ is a least energy
sign-changing solution of system \eqref{1.6} obtained in Theorem
\ref{thm-2}, this together with Lemma \ref{lam-446} implies that $\mathcal{I}_\lambda
(u_\lambda)=m_{\lambda}<\frac{2}{3}S^{\frac{3}{2}}$ and $\mathcal{I}_\lambda
'(u_\lambda)=0$. We split the proof
into three claims which can yield to Theorem \ref{thm-4} directly.

\textbf{Claim 1}: For any sequence $ \{\lambda_n\}$ with
$\lambda_n\rightarrow {0^-}$ as $n\to\infty$, $
\{u_{\lambda_n}\}$ is bounded in $H_r^1(\mathbb{R}^3)$.

For any sequence $ \{\lambda_n\}$, there exists a subsequence of $\{u_{\lambda_n}\}$ such that $\mathcal{I}_{\lambda_n}
(u_{\lambda_n})=m_{\lambda_n}<\frac{2}{3}S^{\frac{3}{2}}$  and $\mathcal{I}_{\lambda_n}'(u_{\lambda_n})=0$. Then, let $n\to\infty$, it follows that
\begin{align*}
\frac{2}{3}S^{\frac{3}{2}}
> \mathcal{I}_{\lambda_n}(u_{\lambda_n})
=\mathcal{I}_{\lambda_n}(u_{\lambda_n})
-\frac{1}{q}\langle
\mathcal{I}_{\lambda_n}'(u_{\lambda_n}),u_{\lambda_n}\rangle
\ge \left(\frac{1}{2}-\frac{1}{q}\right)\|u_{\lambda_n}\|^2,
\end{align*}
which indicates that $\{u_{\lambda_n}\}$ is bounded in
$H_r^1(\mathbb{R}^3)$.

In view of Claim $1$ and Proposition \ref{lam-32},
there exists a subsequence of $\{\lambda_n\}$ satisfying
$\lambda_n\to 0$ as $n\to \infty$, still denoted by
$\{\lambda_n\}$ and there exists $u_0\in H_r^1(\mathbb{R}^3)$
such that, for any $r\in[2,6)$,
\begin{equation}\label{9.18}
u_{\lambda_n}\rightharpoonup u_0\
\mathrm{in}\ H_r^1(\mathbb{R}^3),\ \
u_{\lambda_n}\rightarrow u_0\
\mathrm{in}\ L^r(\mathbb{R}^{3}),\ \
u_{\lambda_n}\rightarrow u_0\
\mathrm{a.e.}\ \mathrm{in}\  \mathbb{R}^{3}.
\end{equation}

\textbf{Claim 2}: $u_0$ is a radial sign-changing solution of
problem \eqref{1.7}.

Since ${u_{\lambda_n}}$ is a least energy radial sign-changing
solution of system \eqref{1.6} with $\lambda=\lambda_n$, then
\begin{align}\label{5.3}
&\int_{\mathbb{R}^{3}}(\nabla u_{\lambda_n}\cdot\nabla
v+u_{\lambda_n}v)\mathrm{d}x
+\lambda_n\int_{\mathbb{R}^{3}}\phi_{u_{\lambda_n}}|u_{\lambda_n}|^3
u_{\lambda_n}v\mathrm{d}x
\nonumber\\&\quad\quad\
=\int_{\mathbb{R}^{3}}|u_{\lambda_n}|^{4}u_{\lambda_n}v\mathrm{d}x
+\int_{\mathbb{R}^{3}}|u_{\lambda_n}|^{q-2}u_{\lambda_n}v\mathrm{d}x,
\end{align}
for any $v\in \mathcal{C}_0^\infty(\mathbb{R}^{3})$. In view of
Hardy-Littlewood-Sobolev inequality (see Proposition
\ref{prohlsi}) and H\"{o}lder inequality, we conclude that
\begin{align*}
\left|\int_{\mathbb{R}^{3}}\phi_{u_{\lambda_n}}
\left|u_{\lambda_n}\right|^3
u_{\lambda_n}v\mathrm{d}x\right|
= & \frac{1}{4\pi}\int_{\mathbb{R}^{3}}\int_{\mathbb{R}^{3}}
\frac{|u_{\lambda_n}(y)|^{5}|u_{\lambda_n}(x)|^3
u_{\lambda_n}(x)v(x)}{|x-y|}\mathrm{d}x\mathrm{d}y
\\ \leq & C_1\left(\int_{\mathbb{R}^{3}}|{u_{\lambda_n}}|^{6}
\mathrm{d}x\right)^{\frac{5}{6}}
\left(\int_{\mathbb{R}^{3}}|u_{\lambda_n}|^{\frac{24}{5}}|v|^{\frac{6}{5}}
\mathrm{d}x\right)^{\frac{5}{6}}
\\ \leq & C_2\|u_{\lambda_n}\|^5\|u_{\lambda_n}\|^4\|v\|
\\ \leq & C_3\|v\|.
\end{align*}
Then combining with \eqref{9.12}, \eqref{9.11} and \eqref{9.18}, as $n\to
\infty$, it deduces  from \eqref{5.3} that
\begin{align}\label{6.46}
\int_{\mathbb{R}^{3}}(\nabla u_0\cdot\nabla v+u_0v)\mathrm{d}x
=\int_{\mathbb{R}^{3}}|u_0|^{4}u_0v\mathrm{d}x
+\int_{\mathbb{R}^{3}}|u_0|^{q-2}u_0v\mathrm{d}x,
\end{align}
for any $v\in \mathcal{C}_0^\infty(\mathbb{R}^{3})$. Thereby, $u_0$ is a weak solution of problem \eqref{1.7}.

In order to finish the proof of Claim $2$, it suffices to prove
$u_0^\pm\neq0$.
Since $u_{\lambda_n}\in\mathcal{M}_{\lambda_n}$, that is,
\begin{equation*}
\left\|u^\pm_{\lambda_n}\right\|^2
+\lambda_n
\int_{\mathbb{R}^{3}}
\phi_{u^\pm_{\lambda_n}}\left|u^\pm_{\lambda_n}\right|^5\mathrm{d}x
+\lambda_n
\int_{\mathbb{R}^{3}}
\phi_{u^\mp_{\lambda_n}}\left|u^\pm_{\lambda_n}\right|^5\mathrm{d}x
=\int_{\mathbb{R}^{3}}\left|u^\pm_{\lambda_n}\right|^6\mathrm{d}x
+\int_{\mathbb{R}^{3}}\left|u^\pm_{\lambda_n}\right|^q\mathrm{d}x.
\end{equation*}
For this, we can directly derive from Lemma \ref{lam-126}$(3)$ that
 $\Lambda_2\le\left\|u_{\lambda_n}^\pm\right\|\le \Lambda_1$ for some $\Lambda_1,\Lambda_2>0$ (independent of $\lambda$ and $n$). This together with $\lambda_n\to 0$ as $n\to \infty$, we infer that
 \begin{equation*}
\left|\lambda_n\int_{\mathbb{R}^{3}}\phi_{u^\pm_{\lambda_n}}
\left|u^\pm_{\lambda_n}\right|^5\mathrm{d}x\right|
\le - C_1 \lambda_n  \left\|u^\pm_{\lambda_n}\right\|^{10}
\le -C_2 \lambda_n \to 0,
\end{equation*}
and
\begin{equation*}
\left|\lambda_n\int_{\mathbb{R}^{3}}\phi_{u^\mp_{\lambda_n}}
\left|u^\pm_{\lambda_n}\right|^5\mathrm{d}x\right|
\le - C_3 \lambda_n
\left\|u^\mp_{\lambda_n}\right\|^5
\left\|u^\pm_{\lambda_n}\right\|^5
\le -C_4 \lambda_n \to 0,
\end{equation*}
which implies that
\begin{equation}\label{8.1}
0<\Lambda_2^2
\leq   \lim_{n\to \infty}\left\|u^\pm_{\lambda_n}\right\|^2
= \lim_{n\to
\infty}\int_{\mathbb{R}^{3}}
\left|u^\pm_{\lambda_n}\right|^6\mathrm{d}x
+\int_{\mathbb{R}^{3}}\left|u^\pm_0\right|^q\mathrm{d}x.
\end{equation}
If $\int_{\mathbb{R}^{3}}|u^\pm_0|^q\mathrm{d}x\neq 0$, it
follows from \eqref{6.46} that $\|u^\pm_0\|^2 \ge
\int_{\mathbb{R}^{3}}|u^\pm_0|^q\mathrm{d}x >0$, which implies
that $u^\pm_0\neq0$. Therefore, the proof of Claim $2$ is
completed. Otherwise, if
$\int_{\mathbb{R}^{3}}|u^+_0|^q\mathrm{d}x=0$ or
$\int_{\mathbb{R}^{3}}|u^-_0|^q\mathrm{d}x=0$,
 which leads to the following three cases:

\vspace{0.25em}
\vspace*{0.125em}

$\bullet$ {Case $i$}: $u^+_0\equiv0$ and $u^-_0\equiv0$;

\vspace{0.125em}

$\bullet$ {Case $ii$}: $u^+_0\not\equiv0$ and $u^-_0\equiv0$;

\vspace{0.125em}

$\bullet$ {Case $iii$}: $u^+_0\equiv0$ and $u^-_0\not\equiv0$.

\vspace{0.25em}
\vspace*{0.125em}

\noindent We will show those cases can not happen.

If Case $i$ happens, it follows from \eqref{2.3} and \eqref{8.1}
that
\begin{equation}\label{8.2}
0<\Lambda_2^2
\le   \lim_{n\to \infty}\left\|u^\pm_{\lambda_n}\right\|^2
= \lim_{n\to
\infty}\int_{\mathbb{R}^{3}}
\left|u^\pm_{\lambda_n}\right|^6\mathrm{d}x
\leq S^{-3}\lim_{n\to \infty}\left\|u^\pm_{\lambda_n}\right\|^6,
\end{equation}
then $\lim_{n\to \infty}\left\|u^\pm_{\lambda_n}\right\|^2\geq
S^{\frac{3}{2}}$. 
By \eqref{8.2} and
$\lambda_n\to 0$ as $n\to \infty$, we conclude that
\begin{align*}
m_0=\lim_{n\to \infty}m_{\lambda_n}
=\lim_{n\to \infty}\mathcal{I}_{\lambda_n}(u_{\lambda_n})
&= \frac{1}{2}\lim_{n\to \infty}\|u_{\lambda_n}\|^2
-\frac{1}{6}\lim_{n\to
\infty}\int_{\mathbb{R}^{3}}|u_{\lambda_n}|^6\mathrm{d}x
\\& = \frac{1}{2}\lim_{n\to \infty}\left\|u_{\lambda_n}^+\right\|^2
+\frac{1}{2}\lim_{n\to \infty}\left\|u_{\lambda_n}^-\right\|^2
\\&\quad-\frac{1}{6}\lim_{n\to
\infty}\int_{\mathbb{R}^{3}}\left|u_{\lambda_n}^+\right|^6\mathrm{d}x
-\frac{1}{6}\lim_{n\to
\infty}\int_{\mathbb{R}^{3}}\left|u_{\lambda_n}^-\right|^6\mathrm{d}x
\\&= \frac{1}{3}\lim_{n\to \infty}\left\|u^+_{\lambda_n}\right\|^2
+\frac{1}{3}\lim_{n\to \infty}\left\|u^-_{\lambda_n}\right\|^2
\\& \geq\frac{2}{3}S^{\frac{3}{2}},
\end{align*}
 which leads to a contradiction with
 $m_0<\frac{2}{3}S^{\frac{3}{2}}$ (obtained by Remark \ref{rem-16}).

If Case $ii$ happens, which leads to
\begin{align}\label{u0-g}
0<\Lambda_2^2
\leq   \lim_{n\to \infty}\left\|u^-_{\lambda_n}\right\|^2
= \lim_{n\to
\infty}\int_{\mathbb{R}^{3}}\left|u^-_{\lambda_n}\right|^6\mathrm{d}x
\leq S^{-3}\lim_{n\to \infty}\left\|u^-_{\lambda_n}\right\|^6,
\end{align}
and
\begin{equation}\label{u0+g}
0<\Lambda_2^2
\leq   \lim_{n\to \infty}\left\|u^+_{\lambda_n}\right\|^2
= \lim_{n\to
\infty}\int_{\mathbb{R}^{3}}\left|u^+_{\lambda_n}\right|^6\mathrm{d}x
+\int_{\mathbb{R}^{3}}\left|u^+_{0}\right|^q\mathrm{d}x.
\end{equation}
Then \eqref{u0-g} shows $\lim_{n\to
\infty}\|u^-_{\lambda_n}\|^2\geq S^{\frac{3}{2}}$. Recall that
$u_0$ is a weak solution of problem \eqref{1.7}, i.e.,
\begin{equation}\label{9.19}
\|u^+_{0}\|^2
= \int_{\mathbb{R}^{3}}|u^+_{0}|^6\mathrm{d}x
+\int_{\mathbb{R}^{3}}|u^+_{0}|^q\mathrm{d}x,
\end{equation}
which together with Br\'{e}zis-Lieb lemma, we obtain that
\eqref{u0+g} is equivalent to
\begin{equation}\label{u0+gb}
\lim_{n\to \infty}\left\|u^+_{\lambda_n}-u^+_{0}\right\|^2
= \lim_{n\to
\infty}\int_{\mathbb{R}^{3}}\left|u^+_{\lambda_n}-u^+_{0}\right|^6\mathrm{d}x
\leq S^{-3}\lim_{n\to \infty}\left\|u^+_{\lambda_n}-u^+_{0}\right\|^6.
\end{equation}
Now, if $\lim_{n\to \infty}
\left\|u^+_{\lambda_n}-u^+_{0}\right\|^2=0$,
i.e., $u^+_{\lambda_n}\to u^+_{0}$ in $H_r^1(\mathbb{R}^3)$, by
\eqref{u0-g}, there holds
\begin{align*}
m_0=\lim_{n\to \infty}m_{\lambda_n}
=\lim_{n\to \infty}\mathcal{I}_{\lambda_n}(u_{\lambda_n})
&=\frac{1}{2}\lim_{n\to \infty}\left\|u^+_{\lambda_n}\right\|^2
-\frac{1}{6}\lim_{n\to \infty}
\int_{\mathbb{R}^{3}}\left|u^+_{\lambda_n}\right|^6\mathrm{d}x
-\frac{1}{q}
\int_{\mathbb{R}^{3}}\left|u^+_{0}\right|^q\mathrm{d}x
\\&\quad+\frac{1}{2}\lim_{n\to \infty}\left\|u^-_{\lambda_n}\right\|^2
-\frac{1}{6}\lim_{n\to \infty}\int_{\mathbb{R}^{3}}\left|u^-_{\lambda_n}\right|^6\mathrm{d}x
\\&=
\mathcal{I}_0(u_0^+)
+\frac{1}{2}\lim_{n\to \infty}\left\|u^-_{\lambda_n}\right\|^2
-\frac{1}{6}\lim_{n\to
\infty}\int_{\mathbb{R}^{3}}\left|u^-_{\lambda_n}\right|^6\mathrm{d}x
\\& = \mathcal{I}_0(u_0^+)+\frac{1}{3}\lim_{n\to
\infty}\left\|u^-_{\lambda_n}\right\|^2
\\ & \geq c_0+\frac{1}{3}S^{\frac{3}{2}},
\end{align*}
 which leads to a contradiction with
 $m_0<c_0+\frac{1}{3}S^{\frac{3}{2}}$. Thus, $\lim_{n\to
 \infty}\left\|u^+_{\lambda_n}-u^+_{0}\right\|^2>0$, and \eqref{u0+gb}
 indicates that $\lim_{n\to
 \infty}\left\|u^+_{\lambda_n}-u^+_{0}\right\|^2\geq S^{\frac{3}{2}}$.
 By Br\'{e}zis-Lieb lemma, \eqref{u0-g}, \eqref{9.19} and \eqref{u0+gb}, then
 \begin{align*}
m_0&=\lim_{n\to \infty}m_{\lambda_n}
=\lim_{n\to \infty}\mathcal{I}_{\lambda_n}(u_{\lambda_n})
\\&= \frac{1}{2}\lim_{n\to \infty}\left\|u^+_{\lambda_n}\right\|^2
-\frac{1}{6}\lim_{n\to \infty}
\int_{\mathbb{R}^{3}}\left|u^+_{\lambda_n}\right|^6\mathrm{d}x
-\frac{1}{q}
\int_{\mathbb{R}^{3}}\left|u^+_{0}\right|^q\mathrm{d}x
\\&\quad+\frac{1}{2}\lim_{n\to \infty}\left\|u^-_{\lambda_n}\right\|^2
-\frac{1}{6}\lim_{n\to \infty}
\int_{\mathbb{R}^{3}}\left|u^-_{\lambda_n}\right|^6\mathrm{d}x
\\&= \frac{1}{2}\lim_{n\to \infty}\left\|u^+_{\lambda_n}\right\|^2
-\frac{1}{6}\lim_{n\to \infty}
\int_{\mathbb{R}^{3}}\left|u^+_{\lambda_n}\right|^6\mathrm{d}x
-\frac{1}{q}
\left(\|u^+_{0}\|^2
-\int_{\mathbb{R}^{3}}|u^+_{0}|^6\mathrm{d}x\right)
 +\frac{1}{3}\lim_{n\to \infty}\left\|u^-_{\lambda_n}\right\|^2
\\&\ge \frac{1}{2}\lim_{n\to \infty}\left\|u^+_{\lambda_n}\right\|^2
-\frac{1}{6}\lim_{n\to \infty}
\int_{\mathbb{R}^{3}}\left|u^+_{\lambda_n}\right|^6\mathrm{d}x
-\frac{1}{2}
\|u^+_{0}\|^2
+\frac{1}{6}\int_{\mathbb{R}^{3}}|u^+_{0}|^6\mathrm{d}x
+\frac{1}{3}\lim_{n\to \infty}\left\|u^-_{\lambda_n}\right\|^2
\\&= \frac{1}{2}\lim_{n\to \infty}\left\|u^+_{\lambda_n}-u^+_{0}\right\|^2
-\frac{1}{6}\lim_{n\to \infty}
\int_{\mathbb{R}^{3}}\left|u^+_{\lambda_n}-u^+_{0}\right|^6\mathrm{d}x
+\frac{1}{3}\lim_{n\to \infty}\left\|u^-_{\lambda_n}\right\|^2
\\& = \frac{1}{3}\lim_{n\to
\infty}
\left\|u^+_{\lambda_n}-u^+_{0}\right\|^2+\frac{1}{3}\lim_{n\to
\infty}
\left\|u^-_{\lambda_n}\right\|^2
\\ & \geq \frac{2}{3}S^{\frac{3}{2}},
\end{align*}
which also leads to a contradiction with
$m_0<\frac{2}{3}S^{\frac{3}{2}}$.

Similarly, we can deduce a contradiction if Case $iii$ happens.


\textbf{Claim 3}: Problem \eqref{1.7} possesses a least energy
radial sign-changing solution $z_0$. Moreover, there exists a
unique pair $(s_{\lambda_n},t_{\lambda_n})\in(0,\infty)\times
(0,\infty)$ such that $s_{\lambda_n}z_0^++t_{\lambda_n}z_0^-\in
\mathcal{M}_{\lambda_n}$ and $(s_{\lambda_n},t_{\lambda_n})\to
(1,1)$ as $n\to \infty$.

Recall from Theorem \ref{thm-5} that problem \eqref{1.7} has a
least energy radial sign-changing solution $z_0$, that is,
$\mathcal{I}_0(z_0)=m_0$ and $\mathcal{I}_0'(z_0)=0$. Using Lemma
\ref{lam-16}, we get that there exists a unique pair
$(s_{\lambda_n},t_{\lambda_n})\in(0,\infty)\times (0,\infty)$
such that $s_{\lambda_n}z_0^++t_{\lambda_n}z_0^-\in
\mathcal{M}_{\lambda_n}$. Next, we aim to show that
$(s_{\lambda_n},t_{\lambda_n})\to (1,1)$ as $n\to \infty$. Since $s_{\lambda_n}z_0^++t_{\lambda_n}z_0^-\in
\mathcal{M}_{\lambda_n}$, then
\begin{align*}
&{s_{\lambda_n}^2}\|z_0^+\|^2
+{\lambda_n}{s_{\lambda_n}^{10}}
\int_{\mathbb{R}^{3}}\phi_{z_0^+}|z_0^+|^5\mathrm{d}x
+\lambda_ns_{\lambda_n}^5t_{\lambda_n}^5
\int_{\mathbb{R}^{3}}\phi_{z_0^-}|z_0^+|^5\mathrm{d}x
\\&\quad\quad\quad \quad
=s_{\lambda_n}^6\int_{\mathbb{R}^{3}}|z_0^+|^6\mathrm{d}x
+s_{\lambda_n}^q\int_{\mathbb{R}^{3}}|z_0^+|^q\mathrm{d}x,\\
&{t_{\lambda_n}^2}\|z_0^-\|^2
+{\lambda_n}{t_{\lambda_n}^{10}}
\int_{\mathbb{R}^{3}}\phi_{z_0^-}|z_0^-|^5\mathrm{d}x
+\lambda_n s_{\lambda_n}^5t_{\lambda_n}^5
\int_{\mathbb{R}^{3}}\phi_{z_0^+}|z_0^-|^5\mathrm{d}x
\\&\quad\quad\quad\quad
=t_{\lambda_n}^6\int_{\mathbb{R}^{3}}|z_0^-|^6\mathrm{d}x
+t_{\lambda_n}^q\int_{\mathbb{R}^{3}}|z_0^-|^q\mathrm{d}x.
\end{align*}
This together with $q\in(2,6)$, we conclude that
$\{s_{\lambda_n}\}$ and $\{t_{\lambda_n}\}$ are bounded in
$\mathbb{R}^+$. Then,
up to a subsequence such that $s_{\lambda_n}\rightarrow s_0$ and
$t_{\lambda_n}\rightarrow t_0$ as $n\rightarrow\infty$, there
hold
\begin{align}
{s_0^2}\|z_0^+\|^2
&= s_0^6\int_{\mathbb{R}^{3}}|z_0^+|^6\mathrm{d}x
+s_0^q\int_{\mathbb{R}^{3}}|z_0^+|^q\mathrm{d}x,\label{5.4}\\
{t_0^2}\|z_0^-\|^2
&= t_0^6\int_{\mathbb{R}^{3}}|z_0^-|^6\mathrm{d}x
+t_0^q\int_{\mathbb{R}^{3}}|z_0^-|^q\mathrm{d}x.\label{5.5}
\end{align}
Recall that $z_0$ is a least energy radial sign-changing solution of
problem \eqref{1.7}, that is,
\begin{align}\label{6.76}
\|z_0^\pm\|^2
=\int_{\mathbb{R}^{3}}|z_0^\pm|^6\mathrm{d}x
+\int_{\mathbb{R}^{3}}|z_0^\pm|^q\mathrm{d}x.
\end{align}
Then, we can derive from \eqref{5.4}$-$\eqref{6.76} that
\begin{align*}
\left(1-s_0^{2-q}\right)\|z_0^+\|^2
&=\left(1-s_0^{6-q}\right)\int_{\mathbb{R}^{3}}|z_0^+|^6\mathrm{d}x,\\
\left(1-t_0^{2-q}\right)\|z_0^-\|^2
&=\left(1-t_0^{6-q}\right)\int_{\mathbb{R}^{3}}|z_0^-|^6\mathrm{d}x.
\end{align*}
From $q\in(2,6)$, we can directly check that $(s_0,t_0)=(1,1)$,
which leads to the Claim $3$.

In view of this, it suffices to prove that $u_0$ obtained in
Claim $2$ is a least energy radial sign-changing solution of
problem \eqref{1.7}. Actually, by Claim $3$ and Lemma
\ref{lam-16}, it yields that
\begin{align*}
\mathcal{I}_0(z_0)
\le \mathcal{I}_0(u_0)
=\lim\limits_{n \rightarrow
\infty}\mathcal{I}_{\lambda_n}(u_{\lambda_n})
&\le \lim\limits_{n \rightarrow \infty}
\mathcal{I}_{\lambda_n}({s_{\lambda_n}z_0^++t_{\lambda_n}z_0^-})
=\mathcal{I}_0({z_0^++z_0^-})
=\mathcal{I}_0(z_0).
\end{align*}
This indicates that $u_0$ is a least energy radial sign-changing
solution of problem \eqref{1.7}. Similar arguments as proof of
Theorem \ref{thm-5}, we obtain that $u_0$ has precisely two nodal
domains. Hence, we complete the proof of Theorem \ref{thm-4}.
\end{proof}

\section*{Conflict of interest statements}
\noindent The authors declare that they have no conflict of
interest.

\section*{Data availability statements}
\noindent Data sharing not applicable to this article as no new
data were created or analyzed in this study.

\section*{Acknowledgements}
\noindent This paper was supported by the National Natural
Science Foundation of China (No. 11971393).


\begin{thebibliography}{99}


\bibitem{AA08} A. Ambrosetti, D. Ruiz,
Multiple bound states for the Schr\"{o}dinger-Poisson problem.
\emph{Commun. Contemp. Math.} \textbf{10} (2008), no. 3,
391--404.



\bibitem{AA12} A. Azzollini, P. d'Avenia,
On a system involving a critically growing nonlinearity.
\emph{J. Math. Anal. Appl.} \textbf{387} (2012), no. 1, 433--438.

\bibitem{AA13} A. Azzollini, P. d'Avenia, V. Luisi,
Generalized Schr\"{o}dinger-Poisson type systems.
\emph{Commun. Pure Appl. Anal.} \textbf{12} (2013), no. 2,
867--879.










\bibitem{BV98}
 V. Benci, D. Fortunato,
 An eigenvalue problem for the Schr\"{o}dinger-Maxwell equations.
\emph{Topol. Methods Nonlinear Anal.} \textbf{11} (1998), no. 2,
283--293.


\bibitem{BV02}
 V. Benci, D. Fortunato,
 Solitary waves of the nonlinear Klein-Gordon equation coupled
 with the Maxwell equations.
\emph{Rev. Math. Phys.} \textbf{14} (2002), no. 4, 409--420.


\bibitem{BH83} H. Br\'{e}zis, E. Lieb,
A relation between pointwise convergence of functions and
convergence of functionals.
\emph{Proc. Amer. Math. Soc.} \textbf{88} (1983), no. 3,
486--490.

\bibitem{BH34} H. Br\'{e}zis, L. Nirenberg,
Positive solutions of nonlinear elliptic equations involving
critical Sobolev exponents.
\emph{Comm. Pure Appl. Math.} \textbf{36} (1983), no. 4,
437--477.

\bibitem{KJB03} K.J. Brown, Y. Zhang,
The Nehari manifold for a semilinear elliptic equation with a
sign-changing weight function.
\emph{J. Differential Equations} \textbf{193} (2003), no. 2,
481--499.

\bibitem{CG36} G. Cerami, S. Solimini, M. Struwe,
Some existence results for superlinear elliptic boundary value
problems involving critical exponents.
\emph{J. Funct. Anal.} \textbf{69} (1986), no. 3, 289--306.


\bibitem{CG10} G. Cerami, G. Vaira,
Positive solutions for some non-autonomous
Schr\"{o}dinger-Poisson systems.
\emph{J. Differential Equations} \textbf{248} (2010), no. 3,
521--543.


\bibitem{CXP21}
X.-P. Chen, C.-L. Tang,
Least energy sign-changing solutions for Schr\"{o}dinger-Poisson
system with critical growth.
\emph{Commun. Pure Appl. Anal.} \textbf{20} (2021), no. 6,
2291--2312.



\bibitem{CXP22}
X.-P. Chen, C.-L. Tang,
Positive and sign-changing solutions for critical
Schr\"{o}dinger-Poisson systems with sign-changing potential.
\emph{Qual. Theory Dyn. Syst.} \textbf{21} (2022), no. 3, Paper
No. 89, 41 pp.



\bibitem{HXM21}
X. He,
Positive solutions for fractional Schr\"{o}dinger-Poisson systems
with doubly critical exponents.
\emph{Appl. Math. Lett.} \textbf{120} (2021), Paper No. 107190, 8
pp.






\bibitem{HH50} H. Hofer,
Variational and topological methods in partially ordered Hilbert
spaces.
\emph{Math. Ann.} \textbf{261} (1982), no. 4, 493--514.

\bibitem{LGB90}
G.B. Li,
Some properties of weak solutions of nonlinear scalar field
equations.
\emph{Ann. Acad. Sci. Fenn. Ser. A I Math.} \textbf{15} (1990),
no. 1, 27--36.

\bibitem{LFY14}
F. Li, Y. Li, J. Shi,
Existence of positive solutions to Schr\"{o}dinger-Poisson type
systems with critical exponent.
\emph{Commun. Contemp. Math.} \textbf{16} (2014), no. 6, 1450036,
28 pp.


\bibitem{LFY17}
Y. Li, F. Li, J. Shi,
Existence and multiplicity of positive solutions to
Schr\"{o}dinger-Poisson type systems with critical nonlocal term.
\emph{Calc. Var. Partial Differential Equations} \textbf{56}
(2017), no. 5, Paper No. 134, 17 pp.



\bibitem{LFY13} F. Li, Q. Zhang,
Existence of positive solutions to the Schr\"{o}dinger-Poisson
system without compactness conditions.
\emph{J. Math. Anal. Appl.} \textbf{401} (2013), no. 2, 754--762.











\bibitem{LE83}
E.H. Lieb,
Sharp constants in the Hardy-Littlewood-Sobolev inequality and
related inequalities.
\emph{Ann. of Math.} \textbf{118} (1983), 349--374.


\bibitem{LE01}
E.H. Lieb, M. Loss,
\emph{Analysis}. Second edition. Graduate Studies in Mathematics,
14. American Mathematical Society, Providence, RI, 2001.




\bibitem{LH16}
H. Liu,
Positive solutions of an asymptotically periodic
Schr\"{o}dinger-Poisson system with critical exponent.
\emph{Nonlinear Anal. Real World Appl.} \textbf{32} (2016),
198--212.


\bibitem{LJ17} J. Liu, J.-F. Liao, C.-L. Tang,
Ground state solution for a class of Schr\"{o}dinger equations
involving general critical growth term.
\emph{Nonlinearity} \textbf{30} (2017), no. 3, 899--911.







\bibitem{MC5} C. Miranda,
Un'osservazione su un teorema di Brouwer.
\emph{Boll. Un. Mat. Ital.} (2) \textbf{3} (1940), 5--7.



\bibitem{PHR51} P.H. Rabinowitz,
Variational methods for nonlinear eigenvalue problems.
Eigenvalues of non-linear problems (Centro Internaz. Mat. Estivo
(C.I.M.E.), III Ciclo, Varenna, 1974), pp. 139--195. Edizioni
Cremonese, Rome, 1974.






\bibitem{RD06} D. Ruiz,
The Schr\"{o}dinger-Poisson equation under the effect of a
nonlinear local term.
\emph{J. Funct. Anal.} \textbf{237} (2006), no. 2, 655--674.



\bibitem{SW15} W. Shuai, Q. Wang,
Existence and asymptotic behavior of sign-changing solutions for
the nonlinear Schr\"{o}dinger-Poisson system in $\mathbb{R}^3$.
\emph{Z. Angew. Math. Phys.} \textbf{66} (2015), no. 6,
3267--3282.



\bibitem{SJT12} J. Sun, H. Chen, J.J. Nieto,
On ground state solutions for some non-autonomous
Schr\"{o}dinger-Poisson systems.
\emph{J. Differential Equations} \textbf{252} (2012), no. 5,
3365--3380.




\bibitem{GT92}
G. Tarantello, Nodal solutions of semilinear elliptic equations
with critical exponent.
\emph{Differential Integral Equations} \textbf{5} (1992), no. 1,
25--42.


\bibitem{WC20}
C. Wang, J. Su,
Critical exponents of weighted Sobolev embeddings for radial
functions.
\emph{Appl. Math. Lett.}  \textbf{107}  (2020), 106484, 6 pp.


\bibitem{WDB19}
 D.-B. Wang, H.-B. Zhang, W. Guan,
 Existence of least-energy sign-changing solutions for
 Schr\"{o}dinger-Poisson system with critical growth.
 \emph{J. Math. Anal. Appl.} \textbf{479} (2019), no. 2,
 2284--2301.


\bibitem{WZP15}
Z. Wang, H.-S. Zhou,
Sign-changing solutions for the nonlinear Schr\"{o}dinger-Poisson
system in $\mathbb{R}^3$.
\emph{Calc. Var. Partial Differential Equations} \textbf{52}
(2015), no. 3--4, 927--943.








\bibitem{WM14} M. Willem,
\emph{Minimax Theorems}. Progress in Nonlinear Differential
Equations and their Applications, 24. Birkh\"{a}user Boston,
Inc., Boston, MA, 1996. 



















\bibitem{LFY20}
L.-F. Yin, X.-P. Wu, C.-L. Tang,
Existence and concentration of ground state solutions for
critical Schr\"{o}dinger-Poisson system with steep potential
well.
\emph{Appl. Math. Comput.} \textbf{374} (2020), 125035, 12 pp.



\bibitem{ZJ15} J. Zhang,
On ground state and nodal solutions of Schr\"{o}dinger-Poisson
equations with critical growth.
\emph{J. Math. Anal. Appl.} \textbf{428} (2015), no. 1, 387--404.


\bibitem{ZZH21}
Z. Zhang, Y. Wang, R. Yuan,
Ground state sign-changing solution for Schr\"{o}dinger-Poisson
system with critical growth.
\emph{Qual. Theory Dyn. Syst.} \textbf{20} (2021), no. 2, Paper
No. 48, 23 pp.




\bibitem{ZXJ18} X.-J. Zhong, C.-L. Tang,
Ground state sign-changing solutions for a
Schr\"{o}dinger-Poisson system with a critical nonlinearity in
$\mathbb{R}^3$.
\emph{Nonlinear Anal. Real World Appl.} \textbf{39} (2018)
166--184.



\end{thebibliography}
\end{document}